\documentclass[11pt,a4paper]{article}
\usepackage[utf8x]{inputenc}
\usepackage{amsmath}
\usepackage{amsfonts}
\usepackage{amsthm}
\usepackage{amssymb}
\usepackage{graphicx}
\usepackage[margin=1in]{geometry}
\usepackage{hyperref}

\usepackage{mathtools}
\usepackage{apptools}
\usepackage[dvipsnames]{xcolor}
\usepackage{enumitem}
\usepackage{mathabx}
\usepackage{cases}
\usepackage{stmaryrd}
\usepackage[framemethod=tikz]{mdframed}
\usepackage{float}

\newtheorem{thm}{Theorem}
\newtheorem{lem}[thm]{Lemma}
\newtheorem{prop}[thm]{Proposition}

\newtheorem{assume}{Assumption}

\theoremstyle{definition}

\newtheorem{rmk}[thm]{Remark}
\newtheorem{ex}[thm]{Example}

\setlist[enumerate]{label=$\rm{(\roman*)}$,leftmargin=\parindent}
\numberwithin{equation}{section}
\numberwithin{thm}{section}
\numberwithin{hypo}{section}
\numberwithin{table}{section}
\numberwithin{figure}{section}

\newcommand{\sR}{\mathbb{R}}
\newcommand{\sL}{\mathbb{L}}
\newcommand{\sX}{\mathcal{X}}
\newcommand{\sY}{\mathcal{Y}}

\newcommand{\sol}{\mathbb{S}}
\newcommand{\Fea}{\mathbb{F}}

\newcommand{\bO}{\mathcal{O}}

\newcommand{\E}{\mathcal{E}}

\newcommand{\Lag}{\mathcal{L}}
\newcommand{\Lb}{\mathcal{L}_{\beta}}

\newcommand{\TL}{\mathcal{T}_{\Lag}}

\title{Time rescaling of a primal-dual dynamical system with asymptotically vanishing damping}
\author{David Alexander Hulett\footnote{Faculty of Mathematics, University of Vienna, Oskar-Morgenstern-Platz 1, 1090 Vienna, Austria, e-mail: \url{david.alexander.hulett@univie.ac.at}.}
\and Dang-Khoa Nguyen\footnote{Faculty of Mathematics, University of Vienna, Oskar-Morgenstern-Platz 1, 1090 Vienna, Austria, e-mail: \url{dang-khoa.nguyen@univie.ac.at}. Research supported by FWF (Austrian Science Fund), project P 34922-N.}}

\usepackage{pdfpages}

\begin{document}

	
\maketitle	
	
\begin{abstract}
In this work, we approach the minimization of a continuously differentiable convex function under linear equality constraints by a second-order dynamical system with an asymptotically vanishing damping term. The system under consideration is a time rescaled version of another system previously found in the literature. We show fast convergence of the primal-dual gap, the feasibility measure, and the objective function value along the generated trajectories. These convergence rates now depend on the rescaling parameter, and thus can be improved by choosing said parameter appropriately. When the objective function has a Lipschitz continuous gradient, we show that the primal-dual trajectory asymptotically converges weakly to a primal-dual optimal solution to the underlying minimization problem. We also exhibit improved rates of convergence of the gradient along the primal trajectories and of the adjoint of the corresponding linear operator along the dual trajectories. Even in the unconstrained case, some trajectory convergence result seems to be new. We illustrate the theoretical outcomes through numerical experiments. 
\end{abstract}
\vspace{0.4cm}

\noindent \textbf{Key Words.} Augmented Lagrangian method, primal-dual dynamical system, damped inertial dynamics, Nesterov's accelerated gradient method, Lyapunov analysis, time rescaling, convergence rate, trajectory convergence
\vspace{1ex}

\noindent \textbf{AMS subject classification.} 37N40, 46N10, 65K10, 90C25
	
\section{Introduction}	
	
\subsection{Problem statement and motivation}	
	
In this paper we will consider the optimization problem
\begin{equation}
	\label{intro:pb}
	\begin{array}{rl}
		\min & f \left( x \right), \\
		\textrm{subject to} 	& Ax = b 
	\end{array}
\end{equation}
where
\begin{equation}
	\label{intro:hypo}
	\begin{cases}
		\sX, \sY \textrm{ are real Hilbert spaces}; \\
		f \colon \sX \to \sR \textrm{ is a continuously differentiable convex function}; \\
		A \colon \sX \to \sY \textrm{ is a continuous linear operator and } b \in \sY; \\
		\textrm{the set } \mathbb{S} \textrm{ of primal-dual optimal solutions of } \eqref{intro:pb} \textrm{ is assumed to be  nonempty}.
	\end{cases}
\end{equation}
This model formulation underlies many important applications in various areas, such as image recovery \cite{Goldstein-ODonoghue-Setzer-Baraniuk}, machine learning \cite{Boyd-et.al,Lin-Li-Fang}, the energy dispatch of power grids \cite{Yi-Hong-Liu:15,Yi-Hong-Liu:16}, distributed optimization \cite{Madan-Lall,Zeng-Yi-Hong-Xie} and network optimization \cite{Shi-Johansson,Zeng-Lei-Chen}.

In recent years, there has been a flurry of research on the relationship between continuous time dynamical systems and the numerical algorithms that arise from their discretizations. 
For the unconstrained optimization problem, it has been known that inertial systems with damped velocities enjoy good convergence properties. 
For a convex, smooth function $f : \mathcal{X} \to \mathbb{R}$, Polyak is the first to consider the \textit{heavy ball with friction} (HBF) dynamics \cite{Polyak,Polyak:book}
\begin{equation} \label{eq:heavy ball with friction}
    \tag{HBF}
    \ddot{x}(t) + \gamma \dot{x}(t) + \nabla f(x(t)) = 0 .
\end{equation}
Alvarez and Attouch continue the line of this study, focusing on inertial dynamics with a fixed viscous damping coefficient \cite{Alvarez,Alvarez-Attouch-Bolte-Redont,Attouch-Goudou-Redont}. 
Later on, Cabot, Engler, and Gadat \cite{Cabot-Engler-Gadat,Cabot-Engler-Gadat-2} consider the system that replaces $\gamma$ with a time dependent damping coefficient $\gamma \left( t \right)$.
In \cite{Su-Boyd-Candes}, Su, Boyd, and Candès showed that it turns out one can achieve fast convergence rates by introducing a time dependent damping coefficient which vanishes in a controlled manner, neither too fast nor too slowly, as $t$ goes to infinity
\begin{equation} \label{eq:AVD}
    \tag{AVD}
    \ddot{x}(t) + \frac{\alpha}{t}\dot{x}(t) + \nabla f(x(t)) = 0. 
\end{equation}
For $\alpha \geq 3$, the authors showed that a solution $x \colon \left[ t_{0}, +\infty \right) \to \mathcal{X}$ to \eqref{eq:AVD} satisfies $f(x(t)) - f(x_{*}) = \mathcal{O}\left(\frac{1}{t^{2}}\right)$ as $t \to +\infty$.
In fact, the choice $\alpha = 3$ provides a continuous limit counterpart to Nesterov's celebrated accelerated gradient algorithm \cite{Nesterov:83,Nesterov:book,FISTA}.
Weak convergence of the trajectories to minimizers of $f$ when $\alpha > 3$ has been shown by Attouch, Chbani, Peypouquet, and Redont in \cite{Attouch-Chbani-Peypoquet-Redont} and May in \cite{May}, together with the improved rates of convergence $f(x(t)) - f(x_{*}) = o\left( \frac{1}{t^{2}}\right)$ as $t\to +\infty$.
In the meantime, similar results for the discrete counterpart were also reported by Chambolle and Dossal in \cite{Chambolle-Dossal}, and by Attouch and Peypouquet in \cite{Attouch-Peypouquet:16} .

In \cite{Attouch-Chbani-Riahi:SIOPT}, Attouch, Chbani, and Riahi proposed an inertial proximal type algorithm, which results from a discretization of the time rescaled \eqref{eq:AVD} system
\[
    \ddot{x}(t) + \frac{\alpha}{t}\dot{x}(t) + \delta(t) \nabla f(x(t)) = 0,
\]
where $\delta \colon \left[ t_{0} , + \infty \right) \to \sR_{+}$ is the time scaling function satisfies a certain growth condition and that also enter into convergence rate statement $f \left( x \left( t \right) \right) - f \left( x_{*} \right) = \bO \left( \frac{1}{t^{2} \delta \left( t \right)} \right)$ as $t \to + \infty$.
The resulting algorithm obtained by the authors is considerably simpler than the founding proximal point algorithm proposed by Güler in \cite{Guler-new}, while providing comparable convergence rates for the functional values.

In order to approach constrained optimization problems, Augmented Lagrangian Method (ALM) \cite{Rockafellar:MOOR} (for linearly constrained problems) and Alternating Direction Method of Multipliers (ADMM) \cite{Gabay-Mercier, Boyd-et.al} (for problems with separable objectives and block variables linearly coupled in the constraints) and some of their variants have been shown to enjoy substantial success. 
Continuous-time approaches for structured convex minimization problems formulated in the spirit of the full splitting paradigm have been recently addressed in \cite{Bot-Csetnek-Laszlo:20} and, closely connected to our approach, in \cite{Zeng-Lei-Chen,He-Hu-Fang,Attouch-Chbani-Fadili-Riahi,BotNguyen}, to which we will have a closer look in Subsection \ref{pdvanishing}.
The temporal discretization resulting from these dynamics gives rise to the numerical algorithm with fast convergence rates \cite{He-Hu-Fang:automatica,He-Hu-Fang:NUMA} and with a convergence guarantee for the generated iterate \cite{Bot-Csetnek-Nguyen:fALM}, without additional assumptions such as strong convexity.

In this paper, we will investigate a second-order dynamical system with asymptotic vanishing damping and time rescaling term, which is associated with the optimization problem \eqref{intro:pb} and formulated in terms of its augmented Lagrangian.
The case when the time rescaling term does not appear has been established in \cite{BotNguyen}.
We show that by introducing this time rescaling function, we are able to derive faster convergence rates for the primal-dual gap, the feasibility measure, and the objective function value along generated trajectories while still maintaining the asymptotic behaviour of the trajectories towards a primal-dual optimal solution. 
On the other hand, this work can also be viewed as an extension of the time rescaling technique derived in \cite{Attouch-Chbani-Riahi:SIOPT,Attouch-Chbani-Riahi:HAL} for the constrained case.
To our knowledge, the trajectory convergence for dynamics with time scaling seems to be new, even in the unconstrained case.


\subsection{Notations and a preliminary result}

For both Hilbert spaces $\sX$ and $\sY$, the Euclidean inner product and the associated norm will be denoted by $\left\langle \cdot , \cdot \right\rangle$ and $\left\lVert \cdot \right\rVert$, respectively. The Cartesian product $\sX \times \sY$ will be endowed with the inner product and the associated norm defined for $\left( x , \lambda \right) , \left( z , \mu \right) \in \sX \times \sY$ as
\begin{equation*}
\left\langle \left( x , \lambda \right) , \left( z , \mu \right) \right\rangle
= \left\langle x , z \right\rangle + \left\langle \lambda , \mu \right\rangle
\qquad \textrm{ and } \qquad
\left\lVert \left( x , \lambda \right) \right\rVert
= \sqrt{\left\lVert x \right\rVert ^{2} + \left\lVert \lambda \right\rVert ^{2}},
\end{equation*}
respectively.

Let $f \colon \sX \to \sR$ be a continuously differentiable convex function such that $\nabla f$ is $\ell-$Lipschitz continuous.
For every $x, y \in \sX$ it holds (see \cite[Theorem 2.1.5]{Nesterov:book})
\begin{equation}
\label{pre:f-bound}
0 \leq \dfrac{1}{2 \ell} \left\lVert \nabla f \left( x \right) - \nabla f \left( y \right) \right\rVert ^{2} \leq f \left( x \right) - f \left( y \right) - \left\langle \nabla f \left( y \right) , x - y \right\rangle \leq \dfrac{\ell}{2} \left\lVert x - y \right\rVert ^{2} .
\end{equation}

\section{The primal-dual dynamical approach}
\label{sec:system}

\subsection{Augmented Lagrangian formulation}

Consider the saddle point problem
\begin{equation}
\label{intro:sd}
\min_{x \in \sX} \max_{\lambda \in \sY} \Lag \left( x , \lambda \right)
\end{equation}
associated to problem \eqref{intro:pb}, where $\Lag \colon \sX \times \sY \to \sR$ denotes the \textit{Lagrangian} function
\begin{equation*}
\Lag \left( x , \lambda \right) := f \left( x \right) + \left\langle \lambda , Ax - b \right\rangle .
\end{equation*}
Under the assumptions \eqref{intro:hypo}, $\Lag$ is convex with respect to $x \in \sX$ and affine with respect to $\lambda \in \sY$. A pair $\left( x_{*} , \lambda_{*} \right) \in \sX \times \sY$ is said to be a \textit{saddle point} of the Lagrangian function $\Lag$ if for every $\left( x , \lambda \right) \in \sX \times \sY$
\begin{equation}
\label{intro:sadde-point}
\Lag \left( x_{*} , \lambda \right) \leq \Lag \left( x_{*} , \lambda_{*} \right) \leq \Lag \left( x , \lambda_{*} \right).
\end{equation}
If $\left( x_{*} , \lambda_{*} \right) \in \sX \times \sY$ is a saddle point of $\Lag$ then $x_{*} \in \sX$ is an optimal solution of \eqref{intro:pb}, and $\lambda_{*} \in \sY$ is an optimal solution of its Lagrange dual problem. If $x_{*} \in \sX$ is an optimal solution of \eqref{intro:pb} and a suitable constraint qualification is fulfilled, then there exists an optimal solution $\lambda_{*} \in \sY$ of the Lagrange dual problem such that $\left( x_{*} , \lambda_{*} \right) \in \sX \times \sY$ is a saddle point of $\Lag$. For details and insights into the topic of constraint qualifications for convex duality we refer to \cite{Bauschke-Combettes:book,Bot:book}.

The set of saddle points of $\Lag$, called also primal-dual optimal solutions of \eqref{intro:pb}, will be denoted by $\sol$ and, as stated in the assumptions, it will be assumed to be nonempty.  The set of feasible points of \eqref{intro:pb} will be denoted by $\Fea := \left\lbrace x \in \sX \colon Ax = b \right\rbrace$ and the optimal objective value of \eqref{intro:pb} by $f_{*}$.

The system of primal-dual optimality conditions for \eqref{intro:pb} reads
\begin{equation}
\label{intro:opt-Lag}
\left( x_{*} , \lambda_{*} \right) \in \sol
\Leftrightarrow \begin{cases}
\nabla_{x} \Lag \left( x_{*} , \lambda_{*} \right) 			& = 0 \\
\nabla_{\lambda} \Lag \left( x_{*} , \lambda_{*} \right) 	& = 0
\end{cases} \Leftrightarrow \begin{cases}
\nabla f \left( x_{*} \right) + A^{*} \lambda_{*} 	& = 0 \\
Ax_{*} - b 													& = 0
\end{cases},
\end{equation}
where $A^{*} : \sY \rightarrow \sX$ denotes the adjoint operator of $A$.

For $\beta \geq 0$, we consider also the augmented Lagrangian $\Lb \colon \sX \times \sY \to \sR$ associated with \eqref{intro:pb} 
\begin{equation}
	\label{intro:aug-Lag}
\Lb \left( x , \lambda \right) := \Lag \left( x , \lambda \right) + \dfrac{\beta}{2} \left\lVert Ax - b \right\rVert ^{2} = f \left( x \right) + \left\langle \lambda , Ax - b \right\rangle + \dfrac{\beta}{2} \left\lVert Ax - b \right\rVert ^{2} .
\end{equation}
For every $(x, \lambda) \in \Fea \times \sY$ it holds
\begin{equation}
\label{intro:Fea:eq}
f \left( x \right) = \Lb \left( x , \lambda \right) = \Lag \left( x , \lambda \right) .
\end{equation}
If $\left( x_{*} , \lambda_{*} \right) \in \sol$, then we have for every $\left( x , \lambda \right) \in \sX \times \sY$
\begin{equation*}
\Lag \left( x_{*} , \lambda \right) = \Lb \left( x_{*} , \lambda \right) 
\leq \Lag \left( x_{*} , \lambda_{*} \right) = \Lb \left( x_{*} , \lambda_{*} \right) 
\leq \Lag \left( x , \lambda_{*} \right) \leq \Lb \left( x , \lambda_{*} \right).
\end{equation*}
In addition, from \eqref{intro:opt-Lag} we have
\begin{equation*} \label{intro:opt-Lb}
    \begin{split}
        \left( x_{*} , \lambda_{*} \right) \in \sol
        &\:\Leftrightarrow\: 
        \begin{cases}
            \nabla f \left( x_{*} \right) + A^{*} \lambda_{*} 	& = 0 \\
            Ax_{*} - b 													& = 0
        \end{cases}
        \:\Leftrightarrow\:
        \begin{cases}
            \nabla f \left( x_{*} \right) + A^{*} \lambda_{*} + \beta A^{*}(Ax_{*} - b) &= 0 \\
            Ax_{*} - b &= 0
        \end{cases} \\
        &\:\Leftrightarrow\:
        \begin{cases}
            \nabla_{x} \mathcal{L}_{\beta}(x_{*}, \lambda_{*}) &= 0 \\
            \nabla_{\lambda} \mathcal{L}_{\beta}(x_{*}, \lambda_{*}) &= 0.
        \end{cases}
    \end{split}
\end{equation*}
In other words, for any $\beta \geq 0$ the sets of saddle points of $\mathcal{L}$ and $\mathcal{L}_{\beta}$ are identical. 

\subsection{The primal-dual asymptotic vanishing damping dynamical system with time rescaling}\label{pdvanishing} \label{subsection:the system}

The dynamical system which we associate to \eqref{intro:pb} and investigate in this paper reads
\begin{mdframed}	
	\begin{equation}
		\label{ds:timerescaled}
		\begin{dcases}
			\ddot{x} \left( t \right) + \dfrac{\alpha}{t} \dot{x} \left( t \right) + \delta \left( t \right) \nabla_{x} \Lb \Bigl( x \left( t \right) , \lambda \left( t \right) + \theta t \dot{\lambda} \left( t \right) \Bigr)  		& = 0 \\
			\ddot{\lambda} \left( t \right) + \dfrac{\alpha}{t} \dot{\lambda} \left( t \right) - \delta \left( t \right) \nabla_{\lambda} \Lb \Bigl( x \left( t \right) + \theta t \dot{x} \left( t \right) , \lambda \left( t \right) \Bigr) 	& = 0 \\
            \Bigl( x \left( t_{0} \right) , \lambda \left( t_{0} \right) \Bigr) 			= \Bigl( x_{0} , \lambda_{0} \Bigr) \textrm{ and }
            \Bigl( \dot{x} \left( t_{0} \right) , \dot{\lambda} \left( t_{0} \right) \Bigr) = \Bigl( \dot{x}_{0} , \dot{\lambda}_{0} \Bigr)
		\end{dcases},
	\end{equation}
\end{mdframed}
where $t_0 >0$, $\alpha > 0$, $\theta > 0$, $\delta \colon \left[ t_{0}, +\infty \right) \to \mathbb{R}$ is a nonnegative continuously differentiable function and $\left( x_{0} , \lambda_{0} \right), \bigl( \dot{x}_{0} , \dot{\lambda}_{0} \bigr) \in \sX \times \sY$ are the initial conditions. Replacing the expressions of the partial gradients of $\Lb$ into the system leads to the following formulation for \eqref{ds:timerescaled}:
\begin{equation*}
\begin{dcases}
\ddot{x} \left( t \right) + \dfrac{\alpha}{t} \dot{x} \left( t \right) + \delta \left( t \right) \nabla f \left( x \left( t \right) \right) + \delta \left( t \right) A^{*} \left( \lambda \left( t \right) + \theta t \dot{\lambda} \left( t \right) \right) + \delta \left( t \right) \beta A^{*} \Bigl( Ax \left( t \right) - b \Bigr)		& = 0 \\
\ddot{\lambda} \left( t \right) + \dfrac{\alpha}{t} \dot{\lambda} \left( t \right) - \delta \left( t \right) \Bigl( A \bigl( x \left( t \right) + \theta t \dot{x} \left( t \right) \bigr) - b \Bigr) 	& = 0 \\
\Bigl( x \left( t_{0} \right) , \lambda \left( t_{0} \right) \Bigr) 			= \Bigl( x_{0} , \lambda_{0} \Bigr) \textrm{ and }
\Bigl( \dot{x} \left( t_{0} \right) , \dot{\lambda} \left( t_{0} \right) \Bigr) = \Bigl( \dot{x}_{0} , \dot{\lambda}_{0} \Bigr).
\end{dcases}
\end{equation*}
The case \eqref{ds:timerescaled} in which there is no time rescaling, i.e., when $\delta(t) \equiv 1$, was studied by Zeng, Lei, and Chen in \cite{Zeng-Lei-Chen}, and by Bo\c{t} and Nguyen in \cite{BotNguyen}.
The system with more general damping, extrapolation and time rescaling coefficients was addressed by He, Hu, and Fang in \cite{He-Hu-Fang,He-Hu-Fang:arXiv} and by Attouch, Chbani, Fadili, and Riahi in \cite{Attouch-Chbani-Fadili-Riahi}.
It is well known that the viscous damping term $\frac{\alpha}{t}$ has a vital role in achieving fast convergence in unconstrained minimization \cite{Attouch-Chbani-Peypouquet-Redont,Attouch-Chbani-Riahi:ESAIM,May}.
The role of the extrapolation $\theta t$ is to induce more flexibility in the dynamical system and in the associated discrete schemes, as it has been recently noticed in \cite{Attouch-Chbani-Fadili-Riahi,Attouch-Chbani-Riahi:Opti,He-Hu-Fang,Zeng-Lei-Chen}.  
The time scaling function $\delta \left( \cdot \right)$ has the role to further improve the rates of convergence of the objective function value along the trajectory, as it was noticed in the context of uncostrained minimization problems in \cite{Attouch-Chbani-Fadili-Riahi:20,Attouch-Chbani-Riahi:SIOPT,Attouch-Chbani-Riahi:PAFA} and of linearly constrained minimization problems in \cite{Attouch-Chbani-Fadili-Riahi,He-Hu-Fang:arXiv}.


Finally, we mention that extending the results in this paper to the multi-block case is possible. For further details, we refer the readers to \cite[Section 2.4]{BotNguyen}.

\subsection{Associated monotone inclusion problem}

The optimality system \eqref{intro:opt-Lag} can be equivalently written  as
\begin{equation}\label{moninclusion}
\TL \left( x_{*} , \lambda_{*} \right) = 0 ,
\end{equation}
where
\begin{equation*}
\TL \colon \sX \times \sY \to \sX \times \sY, \quad \TL \left( x , \lambda \right)
= \begin{pmatrix}
\nabla_{x} \Lag \left( x , \lambda \right) \\ - \nabla_{\lambda} \Lag \left( x , \lambda \right)
\end{pmatrix}
= \begin{pmatrix}
\nabla f \left( x \right) + A^{*} \lambda \\ b-Ax
\end{pmatrix},
\end{equation*}
is the maximally monotone operator associated with the convex-concave function $\Lag$. Indeed, it is immediate to verify that $\TL$ is monotone. Since it is also continuous, it is maximally monotone (see, for instance, \cite[Corollary 20.28]{Bauschke-Combettes:book}). Therefore $\sol$ can be interpreted as the set of zeros of the maximally monotone operator $\TL$, which means that it is a closed convex subset of $\sX \times \sY$ (see, for instance, \cite[Proposition 23.39]{Bauschke-Combettes:book}).

Even in the non-rescaling case, applying the fast continuous-time approaches recently proposed in \cite{Attouch-Peypouquet:19,Attouch} to the solving of \eqref{moninclusion} would require the use of the Moreau-Yosida approximation of the operator $\TL$, for which in general no closed formula is available. The resulting dynamical system would therefore not be formulated in the spirit of the full splitting algorithm, which is undesirable from the point of view of numerical computations.
We mention the work \cite{Bot-Csetnek-Nguyen:fOGDA}, which is related to this time rescaling approach.

\section{Faster convergence rates via time rescaling}\label{sec:Lya}

In this section we will derive fast convergence rates for the primal-dual gap, the feasibility measure, and the objective function value along the trajectories generated by the dynamical system \eqref{ds:timerescaled}. 
We will make the following assumptions on the parameters $\alpha$, $\theta$, $\beta$ and the function $\delta$ throughout this section.
\begin{mdframed}
	\begin{assume}
		\label{assume:para}
		In \eqref{ds:timerescaled}, assume that $\delta : \left[t_{0}, +\infty\right) \to (0, + \infty)$ is continuously differentiable. Moreover, suppose that the parameters $\alpha, \beta$, $\theta$ and the function $\delta$ satisfy 
		\begin{equation}
		\label{cond:gap}
		\alpha \geq 3 , \quad \beta \geq 0 , \quad \dfrac{1}{2} \geq \theta \geq \dfrac{1}{\alpha - 1} \quad \textrm{ and } \quad
		\sup_{t\geq t_{0}}\frac{t \dot{\delta}(t)}{\delta(t)} \leq \frac{1 - 2\theta}{\theta} . 
		\end{equation}
	\end{assume}
\end{mdframed}
Besides the first three conditions that are known previously in \cite{BotNguyen}, it is worth pointing out that we can deduce from the last one the following inequality for every $t \geq t_{0}$:
\begin{equation}
	\label{cond:alpha}
	 \dfrac{t \dot{\delta} \left( t \right)}{\delta \left( t \right)} \leq \dfrac{1 - 2 \theta}{\theta} = \dfrac{1}{\theta} - 2 \leq \alpha - 3 .
\end{equation}
This gives a connection to the condition which appears in \cite{Attouch-Chbani-Riahi:SIOPT}. A few more comments regarding the function $\delta$ will come later, after the convergence rates statements.

\subsection{The energy function}

Let $\left( x , \lambda \right) \colon \left[ t_{0} , + \infty \right) \to \sX \times \sY$ be a solution of \eqref{ds:timerescaled}. Let 
$(x_{*}, \lambda_{*}) \in \mathbb{S}$ be fixed,  we define the energy function $\E \colon \left[ t_{0} , + \infty \right) \to \sR$
\begin{equation}
\label{ener:E}
\E \left( t \right)
:= \theta^{2} t^{2} \delta(t) \left( \Lb \left( x \left( t \right) , \lambda_{*} \right) - \Lb \left( x_{*} , \lambda \left( t \right) \right) \right) + \dfrac{1}{2} \left\lVert v \left( t \right) \right\rVert ^{2} + \dfrac{\xi}{2} \left\lVert \bigl( x \left( t \right) , \lambda \left( t \right) \bigr) - (x_{*}, \lambda_{*}) \right\rVert ^{2} ,
\end{equation}
where
\begin{align}
    v \left( t \right) & := \bigl( x \left( t \right) , \lambda \left( t \right) \bigr) - (x_{*}, \lambda_{*}) + \theta t \left( \dot{x} \left( t \right) , \dot{\lambda} \left( t \right) \right) , \label{ener:v} \\
    \xi  & := \alpha \theta - \theta - 1 \geq 0 . \label{ener:xi}
\end{align}
Notice that due to \eqref{ener:PD-gap}, we have
\begin{equation}
\label{ener:pos}
\E \left( t \right) \geq 0 \qquad \forall t \geq t_{0} .
\end{equation}
In addition, according to \eqref{intro:aug-Lag} and \eqref{intro:Fea:eq}, we have for every $t \geq t_0$
\begin{align}
\Lb \left( x \left( t \right) , \lambda_{*} \right) - \Lb \left( x_{*} , \lambda \left( t \right) \right)
& = \Lag \left( x \left( t \right) , \lambda_{*} \right) - \Lag \left( x_{*} , \lambda \left( t \right) \right) + \dfrac{\beta}{2} \left\lVert Ax \left( t \right) - b \right\rVert ^{2} \label{ener:PD-gap:Lag} \\
& = \Lag \left( x \left( t \right) , \lambda_{*} \right) - f_{*} + \dfrac{\beta}{2} \left\lVert Ax \left( t \right) - b \right\rVert ^{2} \nonumber \\
& = f \left( x \left( t \right) \right) - f_{*} + \left\langle \lambda_{*} , Ax \left( t \right) - b \right\rangle + \dfrac{\beta}{2} \left\lVert Ax \left( t \right) - b \right\rVert ^{2} \geq 0 , \label{ener:PD-gap}
\end{align}
where $f_{*}$ denotes the optimal objective  value of \eqref{intro:pb}.

Assumption \ref{assume:para} implies the nonnegativity of following quantity, which will appear many times in our analysis:
\begin{equation}\label{eq:definition of sigma}
    \sigma \colon \left[ t_{0} , + \infty \right) \to \sR_{+} , \quad \sigma(t) := \frac{1 - 2\theta}{\theta}\delta(t) - t \dot{\delta}(t) . 
\end{equation}

\begin{lem}
	\label{lem:dec}
	Let $\left( x , \lambda \right) \colon \left[ t_{0} , + \infty \right) \to \sX \times \sY$ be a solution of \eqref{ds:timerescaled} and $(x_{*}, \lambda_{*}) \in \mathbb{S}$. For every $t \geq t_{0}$ it holds
	\begin{align*}
	\dfrac{d}{dt} \E \left( t \right) 
	\leq &- \theta^{2} t \sigma(t) \left( \Lb \left( x \left( t \right) , \lambda_{*} \right) - \Lb \left( x_{*} , \lambda \left( t \right) \right) \right) - \dfrac{1}{2} \beta \theta t \delta \left( t \right) \left\lVert Ax \left( t \right) - b \right\rVert ^{2} \nonumber \\
	&- \xi \theta t \left\lVert \left( \dot{x} \left( t \right) , \dot{\lambda} \left( t \right) \right) \right\rVert ^{2} .
	\end{align*}
\end{lem}
\begin{proof}
Let $t \geq t_{0}$ be fixed.
Since $x_{*} \in \Fea$, we have
\begin{align*}
\nabla_{x} \left( \Lb \left( x \left( t \right) , \lambda_{*} \right) - \Lb \left( x_{*} , \lambda \left( t \right) \right) \right) & = \nabla_{x} \Lb \left( x \left( t \right) , \lambda_{*} \right) = \nabla f \left( x \left( t \right) \right) + A^{*} \lambda_{*} + \beta A^{*} \left( Ax \left( t \right) - b \right) , \nonumber \\
\nabla_{\lambda} \left( \Lb \left( x \left( t \right) , \lambda_{*} \right) - \Lb \left( x_{*} , \lambda \left( t \right) \right) \right) & = - \nabla_{\lambda} \Lb \left( x_{*} , \lambda \left( t \right) \right) = 0 .
\end{align*}
Under these expressions, the system \eqref{ds:timerescaled} can be equivalently written as
\begin{align*}
    \left(\ddot{x} \left( t \right) , \ddot{\lambda} \left( t \right)\right)
    =& - \dfrac{\alpha}{t} \left(\dot{x} \left( t \right) , \dot{\lambda} \left( t \right)\right) 
    - \delta \left( t \right) \bigl( \nabla_{x} \Lb \left( x \left( t \right) , \lambda_{*} \right) , 0 \bigr) \nonumber \\
    &- \delta(t)\left(A^{*} \left( \lambda \left( t \right) - \lambda_{*} + \theta t \dot{\lambda} \left( t \right) \right) ,
        - \Bigl( A \left( x \left( t \right) + \theta t \dot{x} \left( t \right) \right) - b \Bigr)\right) ,
\end{align*}
which leads to
\begin{align*}
    \dot{v} \left( t \right)
    =& \: (1 + \theta)\left( \dot{x} \left( t \right) , \dot{\lambda} \left( t \right) \right) + \theta t \left( \ddot{x} \left( t \right) , \ddot{\lambda} \left( t \right) \right) \nonumber \\
    =& \: - \xi \left( \dot{x} \left( t \right) , \dot{\lambda} \left( t \right) \right) - \theta t \delta \left( t \right) \left( t \right) \bigl( \nabla_{x} \Lb \left( x \left( t \right) , \lambda_{*} \right) , 0 \bigr) \nonumber \\
    &\:- \theta t \delta \left( t \right) \begin{pmatrix}
    A^{*} \left( \lambda \left( t \right) - \lambda_{*} + \theta t \dot{\lambda} \left( t \right) \right) ,
    - \Bigl( A \left( x \left( t \right) + \theta t \dot{x} \left( t \right) \right) - b \Bigr)
\end{pmatrix}.
\end{align*}
We get from the distributive property of inner product
\begin{align*}
& \left\langle v \left( t \right) , \dot{v} \left( t \right) \right\rangle \nonumber \\
= 	& - \xi \left\langle \bigl( x \left( t \right) , \lambda \left( t \right) \bigr) - (x_{*}, \lambda_{*}) , \left( \dot{x} \left( t \right) , \dot{\lambda} \left( t \right) \right) \right\rangle 
- \xi \theta t \left\lVert \left( \dot{x} \left( t \right) , \dot{\lambda} \left( t \right) \right) \right\rVert ^{2} \nonumber \\
& - \theta t \delta \left( t \right) \left\langle \left( \nabla_{x} \Lb \left( x \left( t \right) , \lambda_{*} \right) , 0 \right) , \bigl( x \left( t \right) , \lambda \left( t \right) \bigr) - (x_{*}, \lambda_{*}) \right\rangle \nonumber \\
& - \theta^{2} t^{2} \delta \left( t \right) \left\langle \left( \nabla_{x} \Lb \left( x \left( t \right) , \lambda_{*} \right) , 0 \right) , \left( \dot{x} \left( t \right) , \dot{\lambda} \left( t \right) \right) \right\rangle \nonumber \\
& - \theta t \delta \left( t \right) \left\langle \lambda \left( t \right) - \lambda_{*} + \theta t \dot{\lambda} \left( t \right) , Ax \left( t \right) - Ax_{*} \right\rangle 
- \theta^{2} t^{2} \delta \left( t \right) \left\langle \lambda \left( t \right) - \lambda_{*} + \theta t \dot{\lambda} \left( t \right) , A \dot{x} \left( t \right) \right\rangle \nonumber \\
& + \theta t \delta \left( t \right) \left\langle A \left( x \left( t \right) + \theta t \dot{x} \left( t \right) \right) - b , \lambda \left( t \right) - \lambda_{*} \right\rangle
+ \theta^{2} t^{2} \delta \left( t \right) \left\langle A \left( x \left( t \right) + \theta t \dot{x} \left( t \right) \right) - b , \dot{\lambda} \left( t \right) \right\rangle .
\end{align*}
Since $x_{*} \in \Fea$, the last four terms in the above identity vanish. Indeed,
\begin{align*}
& - \left\langle \lambda \left( t \right) - \lambda_{*} + \theta t \dot{\lambda} \left( t \right) , Ax \left( t \right) - Ax_{*} \right\rangle 
- \theta t \left\langle \lambda \left( t \right) - \lambda_{*} + \theta t \dot{\lambda} \left( t \right) , A \dot{x} \left( t \right) \right\rangle \nonumber \\
&  + \left\langle A \left( x \left( t \right) + \theta t \dot{x} \left( t \right) \right) - b , \lambda \left( t \right) - \lambda_{*} \right\rangle 
+ \theta t \left\langle A \left( x \left( t \right) + \theta t \dot{x} \left( t \right) \right) - b , \dot{\lambda} \left( t \right) \right\rangle \nonumber \\
= \ 	& - \left\langle \lambda \left( t \right) - \lambda_{*} + \theta t \dot{\lambda} \left( t \right) , Ax \left( t \right) - b \right\rangle 
- \theta t \left\langle \lambda \left( t \right) - \lambda_{*} + \theta t \dot{\lambda} \left( t \right) , A \dot{x} \left( t \right) \right\rangle \nonumber \\
& + \!\left\langle Ax \left( t \right) - b , \lambda \left( t \right) - \lambda_{*} \right\rangle 
+ \theta t \left\langle A \dot{x} \left( t \right) , \lambda \left( t \right) - \lambda_{*} \right\rangle \nonumber \\
& + \theta t \left\langle Ax \left( t \right) - b , \dot{\lambda} \left( t \right) \right\rangle 
+ \theta^{2} t^{2} \left\langle A \dot{x} \left( t \right) , \dot{\lambda} \left( t \right) \right\rangle \\
=  \ & \ 0 .
\end{align*}
Therefore, differentiating $\mathcal{E}$ with respect to $t$ gives
\begin{align}
    \dfrac{d}{dt} \E \left( t \right) 
    =& \: \theta^{2} t \left( 2 \delta \left( t \right) + t \dot{\delta} \left( t \right) \right) \left( \Lb \left( x \left( t \right) , \lambda_{*} \right) - \Lb \left( x_{*} , \lambda \left( t \right) \right) \right) \nonumber \\
    & + \theta^{2} t^{2} \delta(t) \left\langle \left( \nabla_{x} \Lb \left( x \left( t \right) , \lambda_{*} \right) , 0 \right) , \left( \dot{x} \left( t \right) , \dot{\lambda} \left( t \right) \right) \right\rangle \nonumber \\
    &+ \left\langle v \left( t \right) , \dot{v} \left( t \right) \right\rangle 
    + \xi \left\langle \bigl( x \left( t \right) , \lambda \left( t \right) \bigr) - (x_{*}, \lambda_{*}) , \left( \dot{x} \left( t \right) , \dot{\lambda} \left( t \right) \right) \right\rangle \nonumber \\
    = & \: \theta^{2} t \left( 2 \delta \left( t \right) + t \dot{\delta} \left( t \right) \right) \left( \Lb \left( x \left( t \right) , \lambda_{*} \right) - \Lb \left( x_{*} , \lambda \left( t \right) \right) \right) - \xi \theta t \left\lVert \left( \dot{x} \left( t \right) , \dot{\lambda} \left( t \right) \right) \right\rVert ^{2} \nonumber \\
    & - \theta t \delta(t) \left\langle \left( \nabla_{x} \Lb \left( x \left( t \right) , \lambda_{*} \right) , 0 \right) , \bigl( x \left( t \right) , \lambda \left( t \right) \bigr) - (x_{*}, \lambda_{*}) \right\rangle . \label{dec:pre}
\end{align}
Furthermore, the convexity of $f$ and the fact that $x_{*} \in \Fea$ guarantee
\begin{align}
& - \left\langle \left( \nabla_{x} \Lb \left( x \left( t \right) , \lambda_{*} \right) , 0 \right) , \bigl( x \left( t \right) , \lambda \left( t \right) \bigr) - (x_{*}, \lambda_{*}) \right\rangle \nonumber \\
= \ 	& \left\langle \nabla f \left( x \left( t \right) \right) , x_{*} - x \left( t \right) \right\rangle 
+ \left\langle A^{*} \lambda_{*} , x_{*} - x \left( t \right) \right\rangle
+ \beta \left\langle A^{*} \left( Ax \left( t \right) - b \right) , x_{*} - x \left( t \right) \right\rangle \nonumber \\
\leq \ 	& - \left( f \left( x \left( t \right) \right) - f \left( x_{*} \right) \right) - \left\langle \lambda_{*} , Ax \left( t \right) - b \right\rangle - \beta \left\lVert Ax \left( t \right) - b \right\rVert ^{2} \label{dec:conv} \\
= \ 	& - \left( \Lb \left( x \left( t \right) , \lambda_{*} \right) - \Lb \left( x_{*} , \lambda \left( t \right) \right) \right) - \dfrac{\beta}{2} \left\lVert Ax \left( t \right) - b \right\rVert ^{2} \nonumber,
\end{align}
where we recall that the second equality comes from \eqref{intro:Fea:eq}. By multiplying this inequality by $\theta t \delta(t)$ and combining it with \eqref{dec:pre}, the coefficient attached to the primal-dual gap becomes 
\[
    \theta^{2} t \left( 2\delta(t) + t \dot{\delta}(t)\right) - \theta t \delta(t) = - \theta^{2} t \left( \frac{1 - 2\theta}{\theta} \delta(t) - t\dot{\delta}(t)\right) = - \theta^{2} t \sigma(t),
\]
which finally gives the desired statement. 
\end{proof}

\begin{thm}\label{thm: integral estimates}
    Let $(x, \lambda) : \left[ t_{0}, +\infty \right) \to \mathcal{X} \times \mathcal{Y}$ be a solution of \eqref{ds:timerescaled} and $(x_{*}, \lambda_{*}) \in \mathbb{S}$. The following statements are true: 
    \begin{enumerate}[label = (\roman*)]
        \item it holds 
        \begin{alignat}{2}
            \int_{t_{0}}^{+\infty} t \sigma(t)\Big[\mathcal{L}\left( x(t), \lambda_{*} \right) - \mathcal{L}(x_{*}, \lambda(t))\Big]dt \leq \: & \:\:\E(t_{0}) \:&< +\infty, \label{eq:int estimate 1}\\
            \beta \int_{t_{0}}^{+\infty} t \delta(t) \left\lVert Ax \left( t \right) - b \right\rVert^{2}dt \leq \: &\frac{2\E(t_{0})}{\theta} \:&< +\infty, \label{eq:int estimate 2}\\
            \xi \int_{t_{0}}^{+\infty} t \left\lVert \left( \dot{x} \left( t \right) , \dot{\lambda} \left( t \right) \right) \right\rVert^{2} \leq \: &\frac{\E(t_{0})}{\theta} \: &< +\infty; \label{eq:int estimate 3}
        \end{alignat}
        \item if, in addition, $\alpha > 3$ and $\frac{1}{2} \geq \theta > \frac{1}{\alpha - 1}$, then the trajectory $(x(t), \lambda(t))_{t\geq t_{0}}$ is bounded and the convergence rate of its velocity is 
        \[
        \left\lVert \left( \dot{x} \left( t \right) , \dot{\lambda} \left( t \right) \right) \right\rVert = \mathcal{O}\left(\frac{1}{t}\right) \quad \text{as} \quad t\to +\infty.
        \]
    \end{enumerate}
\end{thm}

\begin{proof}
    (i) Recall that Assumption \ref{assume:para} implies $\sigma(t) \geq 0$ for all $t\geq t_{0}$ and $\xi \geq 0$. Moreover, $(x_{*}, \lambda_{*}) \in \mathbb{S}$ yields $x_{*} \in \mathbb{F}$. Therefore, we can apply  observation \eqref{ener:PD-gap} and Lemma \ref{lem:dec} to obtain, for every $t\geq t_{0}$,
    \begin{align}
        \frac{d}{dt}\E(t) \leq& \: -\theta^{2} t \sigma(t) \left( \Lb \left( x \left( t \right) , \lambda_{*} \right) - \Lb \left( x_{*} , \lambda \left( t \right) \right) \right) \\
        & - \frac{1}{2}\beta \theta t \delta(t) \left\lVert Ax \left( t \right) - b \right\rVert^{2} - \xi \theta t \left\lVert \left( \dot{x} \left( t \right) , \dot{\lambda} \left( t \right) \right) \right\rVert^{2} 
        \leq 0. \label{eq:E not increasing}
    \end{align}
    This means that $\E$ is nonincreasing on $\left[ t_{0}, +\infty \right)$. For every $t\geq t_{0}$, by integrating \eqref{eq:E not increasing} from $t_{0}$ to $t$, we obtain
    \begin{align*}
        &\theta^{2}\int_{t_{0}}^{t} t\sigma(t)\Bigl(\mathcal{L}(x(s), \lambda_{*}) - \mathcal{L}(x_{*}, \lambda(s))\Bigr) ds \\
        & + \frac{\beta \theta}{2}\int_{t_{0}}^{t}s \delta(s)\|Ax(s) - b\|^{2}ds + \xi \theta \int_{t_{0}}^{t}s\left\|\left(\dot{x}(s), \dot{\lambda}(s)\right)\right\|^{2} ds \\
        \leq \:\: &\E(t_{0}) - \E(t) \leq \E(t_{0}),
    \end{align*}
    where the last inequality follows from \eqref{ener:pos}. Since all quantities inside the integrals are nonnegative, we obtain \eqref{eq:int estimate 1}-\eqref{eq:int estimate 3} by letting $t \to +\infty$.
    
    (ii) Inequality \eqref{eq:E not increasing} tells us that $\E$ is nonincreasing on $\left[ t_{0}, +\infty \right)$. Hence, for every $t\geq t_{0}$ it holds 
    \begin{equation}
        \label{eq:energy is bounded}
        \begin{split}
            &\theta^{2} t^{2} \delta(t) \Bigl(\mathcal{L}_{\beta}\left( x(t), \lambda_{*} \right) - \mathcal{L}_{\beta}(x_{*}, \lambda(t))\Bigr) + \left\|v(t)\right\|^{2} + \frac{\xi}{2}\left\|\bigl(x(t), \lambda(t)\bigr) - (x_{*}, \lambda_{*})\right\|^{2} 
            \leq \E(t_{0}).
        \end{split}
    \end{equation}Assuming $\alpha > 3$ and $\frac{1}{2} \geq \theta > \frac{1}{\alpha - 1}$, we immediately see $\xi > 0$. From \eqref{eq:energy is bounded} we obtain, for every $t\geq t_{0}$,
    \begin{equation}
        \label{eq:trajectory bounded}
        \left\|\bigl(x(t), \lambda(t)\bigr) - (x_{*}, \lambda_{*})\right\|^{2} \leq \frac{2\E(t_{0})}{\xi},
    \end{equation}
    which implies the boundedness of the trajectory. On the other hand, the same inequality gives for all $t\geq t_{0}$
    \begin{equation}
        \left\|v(t)\right\| = \left\|\bigl(x(t), \lambda(t)\bigr) - (x_{*}, \lambda_{*}) + \theta t \left(\dot{x}(t), \dot{\lambda}(t)\right)\right\| \leq \sqrt{2\E(t_{0})}.
    \end{equation}
    Using the triangle inequality and \eqref{eq:trajectory bounded}, we obtain for all $t\geq t_{0}$
    \begin{align}
        t\left\lVert \left( \dot{x} \left( t \right) , \dot{\lambda} \left( t \right) \right) \right\rVert &\leq \frac{1}{\theta} \left(\left\|\bigl(x(t), \lambda(t)\bigr) - (x_{*}, \lambda_{*})\right\| + \left\|v\right\|\right) \nonumber\\
        &\leq \frac{1}{\theta}\left(\sqrt{\frac{2\E(t_{0})}{\xi}} + \sqrt{2\E(t_{0})}\right) = \frac{1}{\theta}\left(\frac{1}{\sqrt{\xi}} + 1\right)\sqrt{2\E(t_{0})}, \label{eq:rate 1/t for velocity}
    \end{align}
    which gives the desired convergence rate. 
\end{proof}

\subsection{Fast convergence rates for the primal-dual gap, the feasibility measure and the objective function value}

The following are the main convergence rates results of the paper.
\begin{thm}\label{thm:rates of convergence}
    Let $(x, \lambda) : \left[ t_{0}, +\infty \right) \to \mathcal{X}\times \mathcal{Y}$ be a solution of \eqref{ds:timerescaled} and $(x_{*}, \lambda_{*}) \in \mathbb{S}$. The following statements are true:
    \begin{enumerate}[label = (\roman*)]
        \item for every $t\geq t_{0}$ it holds 
        \begin{equation}
            \label{eq:rate primal-dual gap}
            0 \leq \mathcal{L}\left( x(t), \lambda_{*} \right) - \mathcal{L}(x_{*}, \lambda(t)) \leq \frac{\E(t_{0})}{\theta^{2} t^{2} \delta(t)}; 
        \end{equation}
        
        \item for every $t\geq t_{0}$ it holds
        \begin{equation}\label{eq:feasibility condition}
            \left\lVert Ax \left( t \right) - b \right\rVert \leq \frac{2 C_{1}}{t^{2} \delta(t)}, 
        \end{equation}
        where 
        \[
            C_{1} := \sup_{t\geq t_{0}} t \left\lVert \dot{\lambda} \left( t \right) \right\rVert + (\alpha - 1) \sup_{t\geq t_{0}} \left\lVert \lambda \left( t \right) - \lambda_{*} \right\rVert + t_{0}^{2} \delta(t_{0}) \left\lVert Ax(t_{0}) - b \right\rVert + t_{0}\left\lVert \dot{\lambda} \left( t_{0} \right) \right\rVert;
        \]
        
        \item for every $t\geq t_{0}$ it holds 
        \begin{equation}\label{eq:functional values}
            \left\lvert f \left( x \left( t \right) \right) - f_{*} \right\rvert \leq \left(\frac{\E(t_{0})}{\theta^{2}} + 2 C_{1} \left\lVert \lambda_{*} \right\rVert\right) \frac{1}{t^{2} \delta(t)}. 
        \end{equation}
    \end{enumerate}
\end{thm}
\begin{proof}
    (i) We have already established that $\E$ is nonincreasing on $\left[ t_{0}, +\infty \right)$. Therefore, from the expression for $\E$ and relation \eqref{ener:PD-gap:Lag} we deduce 
    \begin{equation}\label{eq:aux primal dual gap}
        \theta^{2} t^{2} \delta(t) \bigl[\mathcal{L}\left( x(t), \lambda_{*} \right) - \mathcal{L}(x_{*}, \lambda(t)) \bigr] \leq \E(t_{0}) \quad \forall t\geq t_{0},
    \end{equation}
    and the first claim follows.

    (ii) From the second line of \eqref{ds:timerescaled}, for every $t\geq t_{0}$ we have 
    \begin{equation}\label{eq:alt proof rates 1}
        t \ddot{\lambda}(t) + \alpha \dot{\lambda}(t) = t\delta(t) \Bigl(A\bigl(x(t) + \theta t \dot{x}(t)\bigr) - b\Bigr) = t\delta(t)\bigl( Ax(t) - b\bigr) + \theta t^{2} \delta(t) A\dot{x}(t).
    \end{equation}
    Fix $t\geq t_{0}$. On the one hand, integration by parts yields
    \begin{equation}\label{eq:alt proof rates 2}
        \begin{split}
            \int_{t_{0}}^{t}\bigl( s \ddot{\lambda}(s) + \alpha \dot{\lambda}(s)\bigr)ds &= t \dot{\lambda}(t) - t_{0} \dot{\lambda}(t_{0}) - \int_{t_{0}}^{t}\dot{\lambda}(s)ds + \alpha \int_{t_{0}}^{t} \dot{\lambda}(s)ds \\
            &= t \dot{\lambda}(t) - t_{0} \dot{\lambda}(t_{0}) + (\alpha - 1)(\lambda(t) - \lambda(t_{0})).
        \end{split}
    \end{equation}
    On the other hand, again integrating by parts leads to 
    \begin{equation}\label{eq:alr proof rates 3}
        \begin{split}
            \int_{t_{0}}^{t}s^{2} \delta(s) A\dot{x}(s) ds =& \: t^{2} \delta(t) (Ax(t) - b) - t_{0}^{2} \delta(t_{0}) (Ax(t_{0}) - b) \\
            &\: - \int_{t_{0}}^{t}\bigl( 2 s \delta(s) + s^{2} \dot{\delta}(s)\bigr) (Ax(s) - b)ds.
        \end{split}
    \end{equation}
    Now, integrating \eqref{eq:alt proof rates 1} from $t_{0}$ to $t$ and using \eqref{eq:alt proof rates 2} and \eqref{eq:alr proof rates 3} gives us 
    \begin{align}
        &t \dot{\lambda}(t) - t_{0} \dot{\lambda}(t_{0}) + (\alpha - 1)(\lambda(t) - \lambda(t_{0})) \nonumber\\
        = \: &\int_{t_{0}}^{t}s \delta(s) (Ax(s) - b) ds + \theta \int_{t_{0}}^{t} s^{2} \delta(s) A\dot{x}(s) ds \nonumber\\
        = \: & t^{2} \delta(t) (Ax(t) - b) - t_{0}^{2} \delta(t_{0})(Ax(t_{0}) - b) + \int_{t_{0}}^{t} s \bigl[(1 - 2\theta)\delta(s) - \theta s \dot{\delta}(s)\bigr] (Ax(s) - b) ds \nonumber\\
        = \: & t^{2} \delta(t) (Ax(t) - b) - t_{0}^{2} \delta(t_{0})(Ax(t_{0}) - b) + \int_{t_{0}}^{t}\frac{(1 - 2\theta)\delta(s) - \theta s \dot{\delta}(s)}{s \delta(s)} s^{2} \delta(s) (Ax(s) - b) ds.
    \end{align}
    It follows that, for every $t \geq t_{0}$, we have
    \begin{equation}
        \left\| t^{2} \delta(t) (Ax(t) - b) + \int_{t_{0}}^{t}\frac{(1 - 2\theta)\delta(s) - \theta s \dot{\delta}(s)}{s \delta(s)} s^{2} \delta(s) (Ax(s) - b) ds \right\| \leq C_{1},
    \end{equation}
    where 
    \[
        C_{1} = \sup_{t\geq t_{0}} t\left\lVert \dot{\lambda} \left( t \right) \right\rVert + (\alpha - 1) \sup_{t\geq t_{0}}\| \lambda(t) - \lambda \left( t_{0} \right) \| + t_{0}^{2} \delta(t_{0}) \left\lVert Ax(t_{0}) - b \right\rVert + t_{0}\left\lVert \dot{\lambda} \left( t_{0} \right) \right\rVert < + \infty,
    \]
    and this quantity is finite in light of \eqref{eq:rate 1/t for velocity} and \eqref{eq:trajectory bounded}. Now, we set 
    \[
        g(t) := t^{2} \delta(t) \left\lVert Ax \left( t \right) - b \right\rVert, \qquad a(t) := \frac{(1 - 2\theta)\delta(t) - \theta t \dot{\delta}(t)}{t \delta(t)} \qquad \forall t\geq t_{0}
    \]
    and we apply Lemma \ref{lem:lemma a1} to deduce that 
    \begin{equation}\label{eq:aux feasibility condition}
        t^{2} \delta(t) \left\lVert Ax \left( t \right) - b \right\rVert \leq 2C_{1} \quad \forall t\geq t_{0}.
    \end{equation}
    
    (iii) For a fixed $t\geq t_{0}$, we have 
    \[
        \mathcal{L}\left( x(t), \lambda_{*} \right) - \mathcal{L}(x_{*}, \lambda(t)) = f(x(t)) - f(x_{*}) + \langle \lambda_{*}, Ax(t) - b\rangle. 
    \]
    Therefore, from using \eqref{eq:aux feasibility condition} and \eqref{eq:aux primal dual gap} we obtain, for every $t\geq t_{0}$, 
    \begin{align*}
        \left\lvert f \left( x \left( t \right) \right) - f_{*} \right\rvert &\leq \mathcal{L}\left( x(t), \lambda_{*} \right) - \mathcal{L}(x_{*}, \lambda(t)) + \left\lVert \lambda_{*} \right\rVert \left\lVert Ax \left( t \right) - b \right\rVert \\
        &\leq \left(\frac{\E(t_{0})}{\theta^{2}} + 2C_{1} \left\lVert \lambda_{*} \right\rVert\right)\frac{1}{t^{2} \delta(t)},
    \end{align*}
which leads to the last statement.
\end{proof}

Some comments regarding the previous proof and results are in order.

\begin{rmk}
    The proof we provided here is significantly shorter than the one derived in \cite{BotNguyen} thanks to Lemma \ref{lem:lemma a1}. This Lemma is inspired by the one used in \cite{He-Hu-Fang:automatica} for showing the fast convergence to zero of the feasibility measure, although the authors study a different dynamical system. 
    On the other hand, when $\delta \left( t \right) \equiv 1$, the results in \cite{BotNguyen} is more robust than the one we obtain here, as it gives the $\bO \left( \frac{1}{t^{2}} \right)$ rates for the sum of primal-dual gap and feasibility measure, rather than each one individually.
\end{rmk}

\begin{rmk}\label{rmk:rate comparison}
    Here are some remarks comparing our rates of convergence to those in \cite{Attouch-Chbani-Fadili-Riahi, He-Hu-Fang:arXiv}.
    \begin{itemize}
        \item \textit{Primal-dual gap}: According to \eqref{eq:rate primal-dual gap}, the following rate of convergence for the primal-dual is exhibited:
        \[
            \mathcal{L}\bigl(x(t), \lambda_{*}\bigr) - \mathcal{L}\bigl(x_{*}, \lambda(t)\bigr) = \mathcal{O}\left(\frac{1}{t^{2} \delta(t)}\right) \quad \text{as} \quad t\to +\infty,
        \]
        which coincides with the findings of \cite{Attouch-Chbani-Fadili-Riahi, He-Hu-Fang:arXiv}.
        
        \item \textit{Feasibility measure}: According to \eqref{eq:feasibility condition}, we have
        \[
            \left\lVert Ax \left( t \right) - b \right\rVert = \mathcal{O}\left(\frac{1}{t^{2} \delta(t)}\right) \quad \text{as} \quad t\to +\infty,
        \]
        which improves the rate $\mathcal{O}\left( \frac{1}{t \sqrt{\delta(t)}}\right)$ reported in \cite{Attouch-Chbani-Fadili-Riahi, He-Hu-Fang:arXiv}.
        
        \item \textit{Functional values}: Relation \eqref{eq:functional values} tells us that
        \[
            \left\lvert f \left( x \left( t \right) \right) - f_{*} \right\rvert = \mathcal{O}\left(\frac{1}{t^{2} \delta(t)}\right) \quad \text{as} \quad t\to +\infty.
        \]
        In \cite{Attouch-Chbani-Fadili-Riahi}, only the upper bound presents this order of convergence. The lower bound obtained is of order $\mathcal{O}\left(\frac{1}{t \sqrt{\delta(t)}}\right)$ as $t \to +\infty$. In \cite{He-Hu-Fang:arXiv}, there are no comments on the rate attained by the functional values in the case of a general time rescaling parameter. 
    \end{itemize}
\end{rmk}

\begin{rmk} \label{rem:choose alpha small enough}
	To further illustrate, notice that 
    \[
        \delta(t) := \delta_{0} t^{n} \qquad \forall t\geq t_{0}
    \]
    fulfills Assumption \ref{assume:para} provided that $\delta_{0} > 0, 0 \leq n \leq \frac{1 - 2\theta}{\theta} = \frac{1}{\theta} - 2$. 
    Therefore, all the statements derived are above are of order $\bO \left( \frac{1}{t^{1/\theta}} \right)$.
    If we desire to obtain fast convergence rates, we must take $\theta$ small, which in the light of Assumption \ref{assume:para} can be achieved by choosing large enough $\alpha$.
    Such behavior can be seen in the unconstrained case \cite{Attouch-Chbani-Riahi:SIOPT} and other settings \cite{He-Hu-Fang:arXiv,Bot-Csetnek-Nguyen:fOGDA}.
\end{rmk}

\section{Weak convergence of the trajectory to a primal-dual solution}

In this section we will show that the solutions to \eqref{ds:timerescaled} weakly converge to an element of $\mathbb{S}$. 
The fact that $\delta \left( t \right)$ enters the convergence rate statement suggests that one can benefit from this time rescaling function when it is at least nondecreasing on $\left[ t_{0}, +\infty \right)$. We are, in fact, going to need this condition when showing trajectory convergence.
\begin{mdframed}
	\begin{assume}
		\label{assume:para2}
		In \eqref{ds:timerescaled}, assume that $\nabla f$ is $\ell$-Lipschitz continuous for some $\ell > 0$ and that $\delta : [t_{0}, +\infty) \to (0, +\infty)$ is continuously differentiable and nondecreasing. Moreover, suppose that the parameters $\alpha, \beta$, $\theta$ and the function $\delta$ satisfy
		\begin{equation*}
		\alpha > 3 , \quad \beta \geq 0 , \quad \dfrac{1}{2} > \theta > \dfrac{1}{\alpha - 1}, \quad \sup_{t\geq t_{0}}\frac{t \dot{\delta}(t)}{\delta(t)} < \frac{1 - 2\theta}{\theta} .
		\end{equation*}
	\end{assume}
\end{mdframed}

\noindent
    Assumption \ref{assume:para2} entails the existence of $C_{2} > 0$ such that 
    \begin{equation}\label{eq: assumption 2 fact 1}
    \frac{t \dot{\delta}(t)}{\delta(t)} + C_{2} \leq \frac{1 - 2\theta}{\theta} \quad \forall t \geq t_{0}.
    \end{equation}
    and therefore it follows further from the nondecreasing property of $\delta$  that
    \begin{equation}\label{eq: assumption 2 fact 2}
    0 < C_{2} \delta(t_{0}) \leq C_{2} \delta(t) \leq (1 - 2\theta)\delta(t) - \theta t \dot{\delta}(t) \quad \forall t\geq t_{0} .
    \end{equation}
    Moreover, from \eqref{eq: assumption 2 fact 1}, for every $t\geq t_{0}$, we have 
    \[
    0 < C_{2} \leq \frac{1 - 2\theta}{\theta} - \frac{t \dot{\delta}(t)}{\delta(t)} = \frac{\sigma(t)}{\delta(t)}, 
    \]
    which gives 
    \begin{equation} \label{eq:inequality for delta and sigma}
        \delta(t) \leq \frac{\sigma(t)}{C_{2}} \quad\forall t\geq t_{0}.
    \end{equation}
%


We also mention that in the setting of Assumption \ref{assume:para2}, the dynamical system \eqref{ds:timerescaled} has a unique global twice continuously differentiable solution.
The proof follows the same argument as in \cite[Theorem 4.1]{BotNguyen}, which relies on the fact that \eqref{ds:timerescaled} can be rewritten as a first-order dynamical system.
More precisely, $\left( x, \lambda \right) \colon \left[ t_{0}, + \infty \right) \to \mathcal{X} \times \mathcal{Y}$ is a solution to \eqref{ds:timerescaled} if and only if $(x, \lambda, y, \nu) \colon \left[ t_{0}, + \infty \right) \to \mathcal{X} \times \mathcal{Y} \times \mathcal{X} \times \mathcal{Y}$ is a solution to 
\begin{equation*}
    \begin{dcases}
        \bigl(\dot{x}(t), \dot{\lambda}(t), \dot{y}(t), \dot{\nu}(t)\bigr) &= F \left( t, x(t), \lambda(t), y(t), \nu(t) \right) \\
        \left( x(t_{0}), \lambda(t_{0}), y(t_{0}), \nu(t_{0}) \right) &= \left( x_{0}, \lambda_{0}, \dot{x}_{0}, \dot{\lambda}_{0} \right),
    \end{dcases}
\end{equation*}
where $F \colon \left[ t_{0}, + \infty \right) \times \mathcal{X} \times \mathcal{Y} \times \mathcal{X} \times \mathcal{Y} \to \mathcal{X} \times \mathcal{Y} \times \mathcal{X} \times \mathcal{Y}$ is given by 
\begin{align*}
    &F(t, z, \mu, w, \eta) \\
    &:= \left( w, \eta, -\frac{\alpha}{t}w - \delta(t)\left[ \nabla f(z) + A^{*}\bigl( \mu + \theta t \eta\bigr) + \beta A^{*}\bigl( Az - b\bigr)\right], -\frac{\alpha}{t}\eta + \delta(t)\bigl[ A\bigl( z + \theta t w\bigr) - b\bigr]\right).
\end{align*}
We omit the details proof and only state the result in the following theorem.
\begin{thm}
    For every choice of initial conditions 
    \[
        x(t_{0}) = x_{0}, \quad \lambda(t_{0}) = \lambda_{0}, \quad \dot{x}(t_{0}) = \dot{x}_{0} \quad \text{and} \quad \dot{\lambda}(t_{0}) = \dot{\lambda}_{0}
    \]
    the system \eqref{ds:timerescaled} has a unique global twice continuously differentiable solution $(x, \lambda) : \left[ t_{0}, +\infty \right) \to \mathcal{X} \times \mathcal{Y}$.
\end{thm}


The additional Lipschitz continuity condition of $\nabla f$ and the fact that $\delta$ is nondecreasing give rise to the following two essential integrability statements.
\begin{prop}
    Let $(x, \lambda) : \left[ t_{0}, +\infty \right) \to \mathcal{X} \times \mathcal{Y}$ be a solution of \eqref{ds:timerescaled} and $(x_{*}, \lambda_{*}) \in \mathbb{S}$. Then it holds
    \item
    \begin{equation}\label{eq:integral for gradient}
        \int_{t_{0}}^{+\infty}t\delta(t) \left\|\nabla f(x(t)) - \nabla f(x_{*})\right\|^{2} dt < +\infty
    \end{equation}
    and
    \begin{equation}\label{eq:integral for feasibility condition}
        \int_{t_{0}}^{+\infty} t\delta(t) \left\lVert Ax \left( t \right) - b \right\rVert^{2} dt < +\infty.
    \end{equation}
\end{prop}
\begin{proof}
    Thanks to $\nabla f$ being $\ell$-Lipschitz continuous, we can use \eqref{pre:f-bound} to refine relation \eqref{dec:conv} in the proof of Lemma \ref{lem:dec}
    \begin{align*}
        &- \left\langle \left( \nabla_{x} \Lb \left( x \left( t \right) , \lambda_{*} \right) , 0 \right) , \bigl( x \left( t \right) , \lambda \left( t \right) \bigr) - (x_{*}, \lambda_{*}) \right\rangle \\
        = \     & \langle \nabla f(x(t)), x_{*} - x(t)\rangle + \left\langle A^{*}\lambda_{*}, x_{*} - x(t)\right\rangle + \beta\left\langle A^{*}(Ax(t) - b), x_{*} - x(t)\right\rangle \\
        \leq \  & - (f(x(t)) - f(x_{*})) - \frac{1}{2\ell}\left\| \nabla f(x(t)) - \nabla f(x_{*})\right\|^{2} - \langle \lambda_{*}, Ax(t) - b\rangle - \beta \left\lVert Ax \left( t \right) - b \right\rVert^{2} \\
        = \     & - \left( \Lb \left( x \left( t \right) , \lambda_{*} \right) - \Lb \left( x_{*} , \lambda \left( t \right) \right) \right) - \frac{1}{2\ell} \left\| \nabla f(x(t)) - \nabla f(x_{*})\right\|^{2} - \frac{\beta}{2}\left\lVert Ax \left( t \right) - b \right\rVert^{2}.
    \end{align*}
    Consequently, combining this inequality with \eqref{dec:pre} yields, for every $t\geq t_{0}$
    \begin{align*}
        \frac{d}{dt}\E(t) \leq& - \theta^{2} t \sigma(t)  \left( \Lb \left( x \left( t \right) , \lambda_{*} \right) - \Lb \left( x_{*} , \lambda \left( t \right) \right) \right) - \xi \theta t \left\lVert \left( \dot{x} \left( t \right) , \dot{\lambda} \left( t \right) \right) \right\rVert^{2} \\
        &- \frac{\theta t\delta(t)}{2\ell}\left\lVert \nabla f \left( x \left( t \right) \right) - \nabla f \left( x_{*} \right) \right\rVert^{2} - \frac{\theta \beta t \delta(t)}{2}\left\lVert Ax \left( t \right) - b \right\rVert^{2} \\
        \leq& - \frac{\theta t \delta(t)}{2\ell} \left\lVert \nabla f \left( x \left( t \right) \right) - \nabla f \left( x_{*} \right) \right\rVert^{2}. 
    \end{align*}
    Integration of this inequality produces \eqref{eq:integral for gradient}. 
    
    The finiteness of the second integral is only entailed by \eqref{eq:int estimate 2} when $\beta > 0$. For the general case $\beta \geq 0$, we use \eqref{eq:feasibility condition} and the fact that $\delta$ is nondecreasing on $\left[ t_{0}, +\infty \right)$ to obtain 
    \[
        \int_{t_{0}}^{+ \infty} t\delta(t) \left\lVert Ax \left( t \right) - b \right\rVert^{2} dt \leq 4 C_{1}^{2} \int_{t_{0}}^{+ \infty} \frac{1}{t^{3} \delta(t)} dt \leq \frac{4 C_{1}^{2}}{\delta(t_{0})} \int_{t_{0}}^{+ \infty} \dfrac{1}{t^{3}} dt < + \infty,
    \]
    and the proof is complete.
\end{proof}

Now, for a given primal-dual solution $(x_{*}, \lambda_{*}) \in \mathbb{S}$, we define the following mappings on $\left[ t_{0}, +\infty \right)$
\begin{align}
    W(t) &:= \delta(t) \bigl[\mathcal{L}_{\beta}\left( x(t), \lambda_{*} \right) - \mathcal{L}_{\beta}(x_{*}, \lambda(t))\bigr] + \frac{1}{2}\left\lVert \left( \dot{x} \left( t \right) , \dot{\lambda} \left( t \right) \right) \right\rVert^{2} \geq 0, \label{eq:formula for W}\\
    \varphi(t) &:= \frac{1}{2}\left\|\bigl(x(t), \lambda(t)\bigr) - (x_{*}, \lambda_{*})\right\|^{2} \geq 0.
\end{align}

\begin{lem}\label{lem: differential equation for phi}
    Let $(x, \lambda) : \left[ t_{0}, +\infty \right) \to \mathcal{X} \times \mathcal{Y}$ a solution of \eqref{ds:timerescaled} and $(x_{*}, \lambda_{*}) \in \mathbb{S}$. The following inequality holds for every $t\geq t_{0}$:
    \begin{equation}
        \ddot{\varphi}(t) + \frac{\alpha}{t} \dot{\varphi}(t) + \theta t \dot{W}(t) + \frac{\delta(t)}{2\ell}\left\lVert \nabla f \left( x \left( t \right) \right) - \nabla f \left( x_{*} \right) \right\rVert^{2} + \frac{\beta \delta(t)}{2}\left\lVert Ax \left( t \right) - b \right\rVert^{2} \leq 0.
    \end{equation}
\end{lem}

\begin{proof}
    Fix $t\geq t_{0}$. Differentiating $W$ with respect to time yields
    \begin{align}
        \dot{W}(t) = \: &\dot{\delta}(t)\Bigl[\mathcal{L}_{\beta}\left( x(t), \lambda_{*} \right) - \mathcal{L}_{\beta}(x_{*}, \lambda(t))\Bigr] \nonumber\\
        &+ \delta(t) \Bigl[\left\langle \nabla_{x}\mathcal{L}_{\beta}\left( x(t), \lambda_{*} \right), \dot{x}(t)\right\rangle - \left\langle \nabla_{\lambda}\mathcal{L}_{\beta}(x_{*}, \lambda(t)), \dot{\lambda}(t)\right\rangle\Bigr] \nonumber\\
        &+ \langle \ddot{x}(t), \dot{x}(t)\rangle + \langle \ddot{\lambda}(t), \dot{\lambda}(t)\rangle. \nonumber
    \end{align}
    Recall the formulas for the gradients of $\mathcal{L}$:
    \begin{align*}
        \nabla_{x}\mathcal{L}_{\beta}\left( x(t), \lambda_{*} \right) &= \nabla f(x(t)) + A^{*}\lambda_{*} + \beta A^{*}(Ax(t) - b), \\
        \nabla_{\lambda}\mathcal{L}_{\beta}(x_{*}, \lambda(t)) &= Ax_{*} - b = 0,
    \end{align*}
    since $x_{*} \in \mathbb{F}$. 
    Plugging this into the expression for $\dot{W}(t)$ gives us
    \begin{align*}
        \dot{W}(t) = \: &\dot{\delta}(t)\Bigl[\mathcal{L}_{\beta}\left( x(t), \lambda_{*} \right) - \mathcal{L}_{\beta}(x_{*}, \lambda(t))\Bigr] \\
        &+ \delta(t) \left\langle \nabla_{x}\mathcal{L}_{\beta}(x(t), \lambda(t) + \theta t \dot{\lambda}(t)), \dot{x}(t)\right\rangle + \langle \ddot{x}(t), \dot{x}(t)\rangle \\
        &- \delta(t) \left\langle \lambda(t) - \lambda_{*} + \theta t \dot{\lambda}(t), A\dot{x}(t)\right\rangle \\
        &- \delta(t) \left\langle \nabla_{\lambda}\mathcal{L}_{\beta}(x(t) + \theta t \dot{x}(t), \lambda(t)), \dot{\lambda(t)}\right\rangle + \langle \ddot{\lambda}(t), \dot{\lambda}(t)\rangle \\
        &+ \delta(t) \left\langle Ax(t) - b + \theta t A\dot{x}(t), \dot{\lambda}(t)\right\rangle.
    \end{align*}
    By regrouping and using \eqref{ds:timerescaled}, we arrive at 
    \begin{equation}
    \label{eq: form of W prime}
    \begin{split}
        \dot{W}(t) = \: &\dot{\delta}(t)\Bigl[\mathcal{L}_{\beta}\left( x(t), \lambda_{*} \right) - \mathcal{L}_{\beta}(x_{*}, \lambda(t))\Bigr] \\
        &- \frac{\alpha}{t}\left\lVert \dot{x} \left( t \right) \right\rVert^{2} - \frac{\alpha}{t}\|\dot{\lambda}(t)\|^{2} - \delta(t) \left\langle \lambda(t) - \lambda_{*}, A\dot{x}(t)\right\rangle + \delta(t) \bigl\langle Ax(t) - b, \dot{\lambda}(t)\bigr\rangle.
    \end{split}
    \end{equation}
    On the other hand, by the chain rule, we have 
    \begin{align*}
        \dot{\varphi}(t) &= \langle x(t) - x_{*}, \dot{x}(t)\rangle + \bigl\langle \lambda(t) - \lambda_{*}, \dot{\lambda}(t)\bigr\rangle, \\
        \ddot{\varphi}(t) &= \langle x(t) - x_{*}, \ddot{x}(t)\rangle + \left\lVert \dot{x} \left( t \right) \right\rVert^{2} + \bigl\langle \lambda(t) - \lambda_{*}, \ddot{\lambda}(t)\bigr\rangle + \bigl\|\dot{\lambda}(t)\bigr\|^{2}.
    \end{align*}
    By combining these relations, \eqref{ds:timerescaled} and the fact that $x_{*} \in \mathbb{F}$, we get 
    \begin{align}
        \ddot{\varphi}(t) + \frac{\alpha}{t}\dot{\varphi}(t) = \: &\left\langle x(t) - x_{*}, \ddot{x}(t) + \frac{\alpha}{t}\dot{x}(t)\right\rangle + \left\langle \lambda(t) - \lambda_{*}, \ddot{\lambda}(t) + \frac{\alpha}{t}\dot{\lambda}(t)\right\rangle
        + \left\lVert \dot{x} \left( t \right) \right\rVert^{2} + \left\lVert \dot{\lambda} \left( t \right) \right\rVert^{2} \nonumber\\
        = \: &- \bigl\langle x(t) - x_{*}, \delta(t) \nabla_{x}\mathcal{L}_{\beta}(x(t), \lambda(t) + \theta t \dot{\lambda}(t))\bigr\rangle \nonumber\\
        &+ \bigl\langle \lambda(t) - \lambda_{*}, \delta(t) \nabla_{\lambda}\mathcal{L}_{\beta}(x(t) + \theta t \dot{\lambda}(t), \lambda(t))\bigr\rangle + \left\lVert \dot{x} \left( t \right) \right\rVert^{2} + \left\lVert \dot{\lambda} \left( t \right) \right\rVert^{2} \nonumber\\
        = \: &-\bigl\langle x(t) - x_{*}, \delta(t) \nabla_{x} \mathcal{L}_{\beta}\left( x(t), \lambda_{*} \right)\bigr\rangle - \bigl\langle Ax(t) - b, \delta(t) \bigl(\lambda(t) - \lambda_{*} + \theta t \dot{\lambda}(t)\bigr)\bigr\rangle \nonumber\\
        &+ \bigl\langle \lambda(t) - \lambda_{*}, \delta(t) \bigl(Ax(t) - b + \theta t A\dot{x}(t)\bigr)\bigr\rangle + \left\lVert \dot{x} \left( t \right) \right\rVert^{2} + \left\lVert \dot{\lambda} \left( t \right) \right\rVert^{2} \nonumber \\ 
        = \: &-\delta(t) \bigl\langle x(t) - x_{*}, \nabla_{x} \mathcal{L}_{\beta}\left( x(t), \lambda_{*} \right)\bigr\rangle \nonumber\\
        &-\theta t \delta(t) \bigl\langle Ax(t) - b, \dot{\lambda}(t)\bigr\rangle + \theta t \delta(t) \bigl\langle \lambda(t) - \lambda_{*}, A \dot{x}(t)\bigr\rangle \nonumber\\
        &+ \left\lVert \dot{x} \left( t \right) \right\rVert^{2} + \left\lVert \dot{\lambda} \left( t \right) \right\rVert^{2} \label{eq: form of phi}
    \end{align}
    The Lipschitz continuity of $\nabla f$ entails
    \begin{align*}
        &-\left\langle x(t) - x_{*}, \nabla_{x} \mathcal{L}_{\beta}\left( x(t), \lambda_{*} \right) \right\rangle \\
        &= -\langle x(t) - x_{*}, \nabla f(x(t))\rangle - \bigl\langle x(t) - x_{*}, A^{*}\lambda_{*}\bigr\rangle - \beta \left\lVert Ax \left( t \right) - b \right\rVert^{2} \\
        &\leq -(f(x(t)) - f(x_{*})) - \frac{1}{2\ell} \left\lVert \nabla f \left( x \left( t \right) \right) - \nabla f \left( x_{*} \right) \right\rVert^{2} - \langle \lambda_{*}, Ax(t) - b\rangle - \beta\left\lVert Ax \left( t \right) - b \right\rVert^{2} \\
        &= -\Bigl( \mathcal{L}_{\beta}\left( x(t), \lambda_{*} \right) - \mathcal{L}_{\beta}(x_{*}, \lambda(t))\Bigr) - \frac{1}{2\ell}\left\lVert \nabla f \left( x \left( t \right) \right) - \nabla f \left( x_{*} \right) \right\rVert^{2} - \frac{\beta}{2}\left\lVert Ax \left( t \right) - b \right\rVert^{2}.
    \end{align*}
    This, together with \eqref{eq: form of phi}, implies
    \begin{align}
        \ddot{\varphi}(t) + \frac{\alpha}{t}\dot{\varphi}(t) \leq \: &- \delta(t)\Bigl( \mathcal{L}_{\beta}\left( x(t), \lambda_{*} \right) - \mathcal{L}_{\beta}(x_{*}, \lambda(t))\Bigr) - \theta t \delta(t) \bigl\langle Ax(t) - b, \dot{\lambda}(t)\bigr\rangle \nonumber\\
        &+ \theta t \delta(t) \bigl\langle \lambda(t) - \lambda_{*}, A\dot{x}(t)\bigr\rangle + \left\lVert \dot{x} \left( t \right) \right\rVert^{2} + \left\lVert \dot{\lambda} \left( t \right) \right\rVert^{2} \nonumber\\
        &- \frac{\delta(t)}{2\ell}\left\lVert \nabla f \left( x \left( t \right) \right) - \nabla f \left( x_{*} \right) \right\rVert^{2} - \frac{\beta \delta(t)}{2} \left\lVert Ax \left( t \right) - b \right\rVert^{2}. \label{eq: inequality for phi}
    \end{align}
    Multiplying \eqref{eq: form of W prime} by $\theta t > 0$ and then adding the result to \eqref{eq: inequality for phi} yields  
    \begin{align*}
        \ddot{\varphi}(t) + \frac{\alpha}{t}\dot{\varphi}(t) + \theta t \dot{W}(t) = \: &-\bigl(\delta(t) - \theta t \dot{\delta}(t)\bigr) \Bigl(\mathcal{L}_{\beta}\left( x(t), \lambda_{*} \right) - \mathcal{L}_{\beta}(x_{*}, \lambda(t))\Bigr) \\
        &-\frac{\delta(t)}{2\ell}\left\lVert \nabla f \left( x \left( t \right) \right) - \nabla f \left( x_{*} \right) \right\rVert^{2} - \frac{\beta \delta(t)}{2}\left\lVert Ax \left( t \right) - b \right\rVert^{2} \\
        &+(1 - \theta \alpha) \left\lVert \left( \dot{x} \left( t \right) , \dot{\lambda} \left( t \right) \right) \right\rVert^{2} \\
        &\leq -\frac{\delta(t)}{2\ell}\left\lVert \nabla f \left( x \left( t \right) \right) - \nabla f \left( x_{*} \right) \right\rVert^{2} - \frac{\beta \delta(t)}{2}\left\lVert Ax \left( t \right) - b \right\rVert^{2},
    \end{align*}
    where the last inequality follows from Assumption \ref{assume:para2}:
    \begin{align*}
        1 - \theta \alpha & \leq -\theta < 0 , \\
        - \delta(t) + \theta t \dot{\delta}(t) &\leq (2\theta - 1)\delta(t) + \theta t \dot{\delta}(t) \leq 0.
    \end{align*}
    The desired result then follows after some rearranging. 
\end{proof}

The following lemma ensures that the first condition of Opial's Lemma is met. 

\begin{lem}\label{lem: first condition of opial lemma}
    Let $(x, \lambda) : \left[ t_{0}, +\infty \right) \to \mathcal{X} \times \mathcal{Y}$ be a solution to \eqref{ds:timerescaled} and $(x_{*}, \lambda_{*}) \in \mathbb{S}$. Then the positive part $[\dot{\varphi}]_{+}$ of $\dot{\varphi}$ belongs to $\sL^{1}\left[ t_{0}, +\infty \right)$ and the limit $\lim_{t \to +\infty} \varphi(t)$ exists. 
\end{lem}

\begin{proof}
    For any $t\geq t_{0}$, we multiply \eqref{lem: differential equation for phi} by $t$ and drop the last two norm squared terms to obtain 
    \[
        t \ddot{\varphi}(t) + \alpha \dot{\varphi}(t) + \theta t^{2} \dot{W}(t) \leq 0. 
    \]
    Recall from \eqref{eq:formula for W} that for every $t\geq t_{0}$ we have 
    \begin{equation}\label{eq:formula for tW(t)}
        tW(t) = t\delta(t) \bigl[\mathcal{L}_{\beta}\left( x(t), \lambda_{*} \right) - \mathcal{L}_{\beta}(x_{*}, \lambda(t))\bigr] + \frac{t}{2}\left\lVert \left( \dot{x} \left( t \right) , \dot{\lambda} \left( t \right) \right) \right\rVert^{2}.
    \end{equation}
    On the one hand, according to \eqref{eq:int estimate 3}, the second summand of the previous expression belongs to $\sL^{1}\left[ t_{0}, +\infty \right)$. On the other hand, using \eqref{eq:inequality for delta and sigma} and \eqref{eq:int estimate 1}, we assert that
    \begin{equation*}
        \int_{t_{0}}^{+\infty} t \delta(t) \bigl[\mathcal{L}_{\beta}\left( x(t), \lambda_{*} \right) - \mathcal{L}_{\beta}(x_{*}, \lambda(t))\bigr] dt
        \leq \frac{1}{C_{2}} \int_{t_{0}}^{+\infty} t \sigma(t) \bigl[\mathcal{L}_{\beta}\left( x(t), \lambda_{*} \right) - \mathcal{L}_{\beta}(x_{*}, \lambda(t))\bigr] dt
        < +\infty.
    \end{equation*}
    Hence, the first summand of \eqref{eq:formula for tW(t)} also belongs to $ \sL^{1}\left[ t_{0}, +\infty \right)$, which implies that the mapping $t \mapsto tW(t)$ belongs to $\sL^{1}\left[ t_{0}, +\infty \right)$ as well. For achieving the desired conclusion, we make use of Lemma \ref{lem:lemma a4} with $\phi := \varphi$ and $w := \theta W$.
\end{proof}

The following results guarantee that the second condition of Opial's Lemma is also met. 

\begin{lem}\label{lem: technical inequality}
    Let $(x, \lambda) : \left[ t_{0}, +\infty \right) \to \mathcal{X} \times \mathcal{Y}$ be a solution to \eqref{ds:timerescaled} and $(x_{*}, \lambda_{*}) \in \mathbb{S}$. The following inequality holds for every $t\geq t_{0}$
    \begin{align*}
        &\frac{\alpha}{t \delta(t)} \frac{d}{dt}\left\lVert \left( \dot{x} \left( t \right) , \dot{\lambda} \left( t \right) \right) \right\rVert^{2} + 2\left\langle \ddot{x}(t) + \frac{\alpha}{t}\dot{x}(t), A^{*}(\lambda(t) - \lambda_{*})\right\rangle \\ 
        &+\theta \frac{d}{dt}\Bigl(t \delta(t) \left\| A^{*}(\lambda(t) - \lambda_{*})\right\|^{2}\Bigr) + \bigl((1 - \theta)\delta(t) - \theta t \dot{\delta}(t)\bigr) \left\|A^{*}(\lambda(t) - \lambda_{*})\right\|^{2} \\
        \leq \:\: &\delta(t) \Bigl[ 2\| \nabla f(x(t)) - \nabla f(x_{*})\|^{2} + \left(2\beta^{2} \|A\|^{2} + 1\right)\left\lVert Ax \left( t \right) - b \right\rVert^{2}\Bigr].
    \end{align*}
\end{lem}
\begin{proof}
    Let $t \geq t_{0}$ be fixed. From \eqref{ds:timerescaled} and the fact that $A^{*}\lambda_{*} = -\nabla f(x_{*})$, we have
    \begin{align}
    & \delta^{2} \left( t \right) \bigl\| \nabla f(x(t)) - \nabla f(x_{*}) + \beta A^{*}(Ax(t) - b)\bigr\|^{2} \nonumber\\
    = \     & \left\|\ddot{x}(t) + \frac{\alpha}{t}\dot{x}(t) + \delta(t) A^{*}\bigl(\lambda(t) - \lambda_{*} + \theta t \dot{\lambda}(t)\bigr)\right\|^{2} \nonumber\\
    = \     & \left\| \ddot{x}(t) + \frac{\alpha}{t}\dot{x}(t)\right\|^{2} + \delta^{2}(t) \bigl\| A^{*}\bigl( \lambda(t) - \lambda_{*} + \theta t \dot{\lambda}(t)\bigr)\bigr\|^{2} \nonumber\\
    & \quad + 2 \delta(t) \left\langle \ddot{x}(t) + \frac{\alpha}{t}\dot{x}(t), A^{*}\bigl(\lambda(t) - \lambda_{*}\bigr)\right\rangle + 2\theta t \delta(t) \bigl\langle \ddot{x}(t), A^{*} \dot{\lambda}(t)\bigr\rangle + 2\alpha \theta t \delta(t) \bigl\langle \dot{x}(t), A^{*} \dot{\lambda}(t)\bigr\rangle. \label{eq:technical inequality 1}
    \end{align}
    Again using \eqref{ds:timerescaled} yields
    \begin{align}
        \delta^{2} \left( t \right) \left\lVert Ax \left( \left( t \right) \right) - b \right\rVert ^{2} = \: &\left\| \ddot{\lambda}(t) + \frac{\alpha}{t}\dot{\lambda}(t) - \theta t \delta(t) A\dot{x}(t)\right\|^{2} \nonumber\\
        = \: &\left\| \ddot{\lambda}(t) + \frac{\alpha}{t}\dot{\lambda}(t)\right\|^{2} + \theta^{2} t^{2} \delta^{2} \left( t \right) \left\lVert A \dot{x} \left( t \right) \right\rVert^{2} \nonumber\\
        & -2\theta t \delta(t) \bigl\langle \ddot{\lambda}(t), A\dot{x}(t)\bigr\rangle - 2\alpha\theta \delta(t) \bigl\langle \dot{\lambda}(t), A\dot{x}(t)\bigr\rangle \label{eq:technical inequality 2}
    \end{align}
    Adding \eqref{eq:technical inequality 1} and \eqref{eq:technical inequality 2} together produces
    \begin{align}
        & \delta^{2} \left( t \right) \bigl\| \nabla f(x(t)) - \nabla f(x_{*}) + \beta A^{*}(Ax(t) - b)\bigr\|^{2} + \delta^{2} \left( t \right) \left\lVert Ax \left( \left( t \right) \right) - b \right\rVert ^{2} \nonumber\\
        &= \left\|\bigl(\ddot{x}(t), \ddot{\lambda}(t)\bigr) + \frac{\alpha}{t}\bigl(\dot{x}(t), \dot{\lambda}(t)\bigr)\right\|^{2} + \delta^{2}(t)\bigl\| A^{*}\bigl( \lambda(t) - \lambda_{*} + \theta t \dot{\lambda}(t)\bigr)\bigr\|^{2} + \theta^{2} t^{2} \delta^{2}(t) \left\lVert A \dot{x} \left( t \right) \right\rVert^{2} \nonumber\\
        &\quad + 2 \theta t \delta(t) \bigl\langle \ddot{x}(t), A^{*}\dot{\lambda}(t)\bigr\rangle - 2 \theta t \delta(t) \bigl\langle \ddot{\lambda}(t), A\dot{x}(t)\bigr\rangle + 2 \delta(t) \left\langle \ddot{x}(t) + \frac{\alpha}{t}\dot{x}(t), A^{*}(\lambda(t) - \lambda_{*})\right\rangle. \label{eq:technical inequality 3}
        \end{align}
        On the one hand, we have 
        \begin{align}
            &\theta^{2} t^{2} \delta^{2}(t) \left\lVert A \dot{x} \left( t \right) \right\rVert^{2} + 2\theta t \delta(t) \bigl\langle \ddot{x}(t), A^{*}\dot{\lambda}(t)\bigr\rangle - 2 \theta t \delta(t) \bigl\langle \ddot{\lambda}(t), A\dot{x}(t)\bigr\rangle \nonumber\\
            = \     & \theta^{2} t^{2} \delta^{2}(t) \bigl\| \bigl(A^{*}\dot{\lambda}(t), -A\dot{x}(t)\bigr)\bigr\|^{2} - \theta^{2} t^{2} \delta^{2}(t) \bigl\|A^{*} \dot{\lambda}(t)\bigr\|^{2} + 2 \theta t \delta(t) \bigl\langle \bigl(\ddot{x}(t), \ddot{\lambda}(t)\bigr), \bigl( A^{*}\dot{\lambda}(t), - A\dot{x}(t)\bigr)\bigr\rangle \nonumber\\
            = \     & - \bigl\|\bigl(\ddot{x}(t), \ddot{\lambda}(t)\bigr)\bigr\|^{2} + \bigl\|\bigl(\ddot{x}(t), \ddot{\lambda}(t)\bigr) + \theta t \delta(t) A^{*}\bigl( A^{*}\dot{\lambda}(t), -A\dot{x}(t)\bigr)\bigr\|^{2} - \theta^{2} t^{2} \delta^{2}(t) \bigl\|A^{*}\dot{\lambda}(t)\bigr\|^{2} \nonumber\\
            \geq \  & - \bigl\| \bigl(\ddot{x}(t), \ddot{\lambda}(t)\bigr)\bigr\|^{2} - \theta^{2} t^{2} \delta^{2}(t) \bigl\|A^{*}\dot{\lambda}(t)\bigr\|^{2}. \label{eq:technical inequality 4}
        \end{align}
        On the other hand, it holds 
        \begin{align}
            &\left\| \bigl(\ddot{x}(t), \ddot{\lambda}(t)\bigr) + \frac{\alpha}{t}\bigl(\dot{x}(t), \dot{\lambda}(t)\bigr)\right\|^{2} - \bigl\|\bigl(\ddot{x}(t), \ddot{\lambda}(t)\bigr)\bigr\|^{2} \nonumber\\
            = \     & \frac{\alpha^{2}}{t^{2}} \bigl\|\bigl(\dot{x}(t), \dot{\lambda}(t)\bigr)\bigr\|^{2} + 2\frac{\alpha}{t}\bigl\langle \bigl(\ddot{x}(t), \ddot{\lambda}(t)\bigr), \bigl( \dot{x}(t), \dot{\lambda}(t)\bigr)\bigr\rangle \geq \frac{\alpha}{t} \frac{d}{dt}\bigl\| \bigl( \dot{x}(t), \dot{\lambda}(t)\bigr)\bigr\|^{2}. \label{eq:technical inequality 5}
        \end{align}
        Moreover, 
        \begin{align}
            &\delta(t)\bigl\| A^{*}\bigl( \lambda(t) - \lambda_{*} + \theta t \dot{\lambda}(t)\bigr)\bigr\|^{2} - \theta^{2} t^{2} \delta(t) \bigl\|A^{*} \dot{\lambda}(t)\bigr\|^{2} \nonumber\\
            = \     & \delta(t)\bigl\| A^{*}( \lambda(t) - \lambda_{*})\bigr\|^{2} + 2\theta t \delta(t) \bigl\langle A A^{*}( \lambda(t) - \lambda_{*}), \dot{\lambda(t)}\bigr\rangle \nonumber\\
            = \     & \bigl((1 - \theta)\delta(t) - \theta t \dot{\delta}(t)\bigr) \bigl\| A^{*} \bigl( \lambda(t) - \lambda_{*}\bigr)\bigr\|^{2} + \theta \delta(t) \bigl\| A^{*}( \lambda(t) - \lambda_{*})\bigr\|^{2} \nonumber\\
            &\quad + \theta t \dot{\delta}(t) \left\lVert A^{*} \left( \lambda \left( t \right) - \lambda_{*} \right) \right\rVert^{2} + \theta \frac{d}{dt} \bigl\| A^{*}(\lambda(t) - \lambda_{*})\|^{2} \nonumber\\
            = \     & \bigl((1 - \theta)\delta(t) - \theta t \dot{\delta}(t)\bigr) \left\lVert A^{*} \left( \lambda \left( t \right) - \lambda_{*} \right) \right\rVert^{2} + \theta \frac{d}{dt}\Bigl(t \delta(t) \bigl\|A^{*}(\lambda(t) - \lambda_{*})\bigr\|^{2}\Bigr). \label{eq:technical inequality 6}
        \end{align}
        Now, using \eqref{eq:technical inequality 4}, \eqref{eq:technical inequality 5} and \eqref{eq:technical inequality 6} in \eqref{eq:technical inequality 3} yields 
        \begin{align*}
            & \delta^{2} \left( t \right) \left( t \right) \bigl\| \nabla f(x(t)) - \nabla f(x_{*}) + \beta A^{*}(Ax(t) - b)\bigr\|^{2} + \delta^{2} \left( t \right) \left\lVert Ax \left( \left( t \right) \right) - b \right\rVert ^{2} \\
            \geq \  & \left\|\bigl(\ddot{x}(t), \ddot{\lambda}(t)\bigr) + \frac{\alpha}{t}\bigl(\dot{x}(t), \dot{\lambda}(t)\bigr)\right\|^{2} + \delta^{2}(t)\bigl\| A^{*}\bigl( \lambda(t) - \lambda_{*} + \theta t \dot{\lambda}(t)\bigr)\bigr\|^{2} \\
            &\quad - \bigl\| \bigl( \ddot{x}(t), \ddot{\lambda}(t)\bigr)\bigr\|^{2} - \theta^{2} t^{2} \delta^{2}(t) \bigl\|A^{*}\dot{\lambda}(t)\bigr\|^{2} + 2 \delta(t) \left\langle \ddot{x}(t) + \frac{\alpha}{t}\dot{x}(t), A^{*}(\lambda(t) - \lambda_{*})\right\rangle \\
            \geq \  & \frac{\alpha}{t}\frac{d}{dt}\bigl\|\bigl(\dot{x}(t), \dot{\lambda}(t)\bigr)\bigr\|^{2} + 2\delta(t) \left\langle \ddot{x}(t) + \frac{\alpha}{t}\dot{x}(t), A^{*}(\lambda(t) - \lambda_{*})\right\rangle \\ 
            &\quad +\delta(t) \left[ \bigl((1 - \theta)\delta(t) - \theta t \dot{\delta}(t)\bigr)\bigl\|A^{*}(\lambda(t) - \lambda_{*})\bigr\|^{2} + \theta \frac{d}{dt} \Bigl( t \delta(t) \left\lVert A^{*} \left( \lambda \left( t \right) - \lambda_{*} \right) \right\rVert^{2}\Bigr)\right] 
        \end{align*}
        Finally, since
        \begin{align*}
            &\bigl\| \nabla f(x(t)) - \nabla f(x_{*}) + \beta A^{*}(Ax(t) - b)\bigr\|^{2} + \left\lVert Ax \left( \left( t \right) \right) - b \right\rVert ^{2} \\
            \leq \  & 2\left\lVert \nabla f \left( x \left( t \right) \right) - \nabla f \left( x_{*} \right) \right\rVert^{2} + (2\beta \|A\|^{2} + 1) \left\lVert Ax \left( t \right) - b \right\rVert^{2}, 
        \end{align*}
        the conclusion follows after dividing the inequality by $\delta(t)$.
\end{proof}

The following proposition provides us with the main integrability result that will be used for verifying the second condition of Opial's Lemma. 
\begin{prop}
    Let $(x, \lambda) : \left[ t_{0}, +\infty \right) \to \mathcal{X} \times \mathcal{Y}$ be a solution to \eqref{ds:timerescaled} and $(x_{*}, \lambda_{*}) \in \mathbb{S}$. Then it holds 
    \begin{equation}\label{eq:integral for dual variable}
        \int_{t_{0}}^{+\infty} t \delta(t) \left\lVert A^{*} \left( \lambda \left( t \right) - \lambda_{*} \right) \right\rVert^{2} dt < +\infty.
    \end{equation}
\end{prop}
\begin{proof}
    Recall that from Lemma \ref{lem: differential equation for phi}, for every $t\geq t_{0}$, we have
    \[
    \ddot{\varphi}(t) + \frac{\alpha}{t}\dot{\varphi}(t) + \theta t \dot{W}(t) \leq - \frac{\delta(t)}{2\ell}\left\lVert \nabla f \left( x \left( t \right) \right) - \nabla f \left( x_{*} \right) \right\rVert^{2} - \frac{\beta \delta(t)}{2}\left\lVert Ax \left( t \right) - b \right\rVert^{2}.
    \]
    Now, to this inequality we add the one produced by Lemma \ref{lem: technical inequality}. For every $t\geq t_{0}$, it holds 
    \begin{align}
        &\ddot{\varphi}(t) + \frac{\alpha}{t}\dot{\varphi}(t) + \theta t \dot{W}(t) + \frac{\alpha}{t\delta(t)}\bigl\|\bigl(\dot{x}(t), \dot{\lambda}(t)\bigr)\bigr\|^{2} \nonumber\\
        &+ \theta \frac{d}{dt}\Bigl(t\delta(t) \bigl\|A^{*}(\lambda(t) - \lambda_{*})\bigr\|^{2}\Bigr) + 2\left\langle \ddot{x}(t) + \frac{\alpha}{t}\dot{x}(t), A^{*}(\lambda(t) - \lambda_{*})\right\rangle \nonumber\\
        \leq \: &-\bigl( (1 - \theta)\delta(t) - \theta t \dot{\delta}(t)\bigr)\bigl\|A^{*}(\lambda(t) - \lambda_{*})\bigr\|^{2} + \left(2 - \frac{1}{2\ell}\right) \delta(t) \left\lVert \nabla f \left( x \left( t \right) \right) - \nabla f \left( x_{*} \right) \right\rVert^{2} \nonumber\\
        &+\left(2\beta^{2} \|A\|^{2} + 1 - \frac{\beta}{2}\right) \delta(t) \left\lVert Ax \left( t \right) - b \right\rVert^{2} \nonumber\\
        \leq \: &-\bigl( (1 - \theta)\delta(t) - \theta t \dot{\delta}(t)\bigr)\bigl\|A^{*}(\lambda(t) - \lambda_{*})\bigr\|^{2} + C_{3} \delta(t) \left\lVert \nabla f \left( x \left( t \right) \right) - \nabla f \left( x_{*} \right) \right\rVert^{2} \nonumber\\
        &+C_{4} \delta(t) \left\lVert Ax \left( t \right) - b \right\rVert^{2}, \label{eq: inequality we will have to integrate}
    \end{align}
    where
    \[
    C_{3} := \left[2 - \frac{1}{2\ell}\right]_{+} \geq 0 \quad \text{and} \quad C_{4} := \left[2\beta^{2}\|A\|^{2} + 1 - \frac{\beta}{2}\right]_{+} \geq 0.
    \]
    Mutiplying \eqref{eq: inequality we will have to integrate} by $t^{\alpha}$ and integrating from $t_{0}$ to $t$, we obtain
    \begin{align}
        &I_{1}(t) + \theta I_{2}(t) + \alpha I_{3}(t) + \theta I_{4}(t) + 2 I_{5}(t) \nonumber\\
        &\leq \int_{t_{0}}^{t}s^{\alpha} \bigl((1 - \theta)\delta(s) - \theta s \dot{\delta}(s)\bigr) \bigl\|A^{*}(\lambda(s) - \lambda_{*})\bigr\|^{2} ds + C_{3} \int_{t_{0}}^{t} s^{\alpha} \delta(s) \|\nabla f(x(s)) - \nabla f(x_{*})\|^{2} ds \nonumber\\
        &\quad + C_{4}\int_{t_{0}}^{t}s^{\alpha} \delta(s) \|Ax(s) - b\|^{2} ds, \label{eq: inequality for I_1, I_2, I_3, I_4, and I_5}
    \end{align}
    where 
    \begin{align*}
        I_{1}(t) &:= \int_{t_{0}}^{t}\bigl(s^{\alpha} \ddot{\varphi}(s) + \alpha s^{\alpha - 1}\dot{\varphi}(s)\bigr) ds, \\
        I_{2}(t) &:= \int_{t_{0}}^{t} s^{\alpha + 1}\dot{W}(s) ds, \\
        I_{3}(t) &:= \int_{t_{0}}^{t} \frac{s^{\alpha - 1}}{\delta(s)} \frac{d}{ds} \bigl\|\bigl(\dot{x}(s), \dot{\lambda}(s)\bigr)\bigr\|^{2} ds, \\
        I_{4}(t) &:= \int_{t_{0}}^{t} s^{\alpha} \frac{d}{ds}\Bigl(s\delta(s) \bigl\|A^{*}(\lambda(s) - \lambda_{*})\bigr\|^{2}\Bigr) ds, \\
        I_{5}(t) &:= \int_{t_{0}}^{t}\bigl\langle s^{\alpha} \ddot{x}(s) + \alpha s^{\alpha - 1}\dot{x}(s), A^{*}(\lambda(s) - \lambda_{*})\bigr\rangle ds.
    \end{align*}
    We will furnish five different inequalities from computing each of these integrals separately. Let $t\geq t_{0}$ be fixed. 
    
    $\bullet$ \textit{The integral $I_{1}(t)$}. By the chain rule, for $s \geq t_{0}$ it holds 
    \[
    s^{\alpha}\ddot{\varphi}(s) + \alpha s^{\alpha - 1}\dot{\varphi}(s) = \frac{d}{ds}\bigl( s^{\alpha} \dot{\varphi}(s)\bigr), 
    \]
    which leads to 
    \begin{equation}\label{eq: inequality for I_1}
    0 = I_{1}(t) - t^{\alpha}\dot{\varphi}(t) + t_{0}^{\alpha}\dot{\varphi}(t_{0}) \leq I_{1}(t) - t^{\alpha}\dot{\varphi}(t) + |t_{0}^{\alpha}\dot{\varphi}(t_{0})|.
    \end{equation}
    
    $\bullet$ \textit{The integral $I_{2}(t)$}. Integration by parts gives 
    \[
    I_{2}(t) = t^{\alpha + 1}W(t) - t_{0}^{\alpha + 1}W(t_{0}) - (\alpha + 1)\int_{t_{0}}^{t}s^{\alpha}W(s)ds, 
    \]
    which implies
    \begin{equation}\label{eq: inequality for I_2}
        0 \leq t^{\alpha + 1}W(t) = I_{2}(t) + t_{0}^{\alpha + 1}W(t_{0}) + (\alpha + 1)\int_{t_{0}}^{t}s^{\alpha}W(s)ds.
    \end{equation}
    
    $\bullet$ \textit{The integral $I_{3}(t)$}. Again by integrating by parts, we get 
    \begin{align*}
        I_{3}(t) = \: &\frac{t^{\alpha - 1}}{\delta(t)} \bigl\|\bigl(\dot{x}(t), \dot{\lambda}(t)\bigr)\bigr\|^{2} - \frac{t_{0}^{\alpha - 1}}{\delta(t_{0})} \bigl\|\bigl(\dot{x}(t_{0}), \dot{\lambda}(t_{0})\bigr)\bigr\|^{2} \\
        &- \int_{t_{0}}^{t} \left[ \frac{(\alpha - 1) s^{\alpha - 2} \delta(s) - s^{\alpha - 1}\dot{\delta}(s)}{\delta^{2}(s)}\right] \bigl\|\bigl(\dot{x}(s), \dot{\lambda}(s)\bigr)\bigr\|^{2} ds. 
    \end{align*}
    For $s\geq t_{0}$, according to Assumption \ref{assume:para2} we have $\dot{\delta}(s) \geq 0$, hence $\delta$ is monotonically increasing and therefore
    \[
    \frac{(\alpha - 1) s^{\alpha - 2} \delta(s) - s^{\alpha - 1} \delta(s)}{\delta^{2}(s)} \leq \frac{(\alpha - 1)s^{\alpha - 2} \delta(s)}{\delta^{2}(s)} \leq \frac{(\alpha - 1)s^{\alpha}}{t_{0}^{2} \delta(t_{0})}.
    \]
    It follows that
    \begin{align}
        0 &\leq \frac{t^{\alpha - 1}}{\delta(t)}\bigl\|\bigl(\dot{x}(t), \dot{\lambda}(t)\bigr)\bigr\|^{2} \nonumber\\
        &= I_{3}(t) + \frac{t_{0}^{\alpha - 1}}{\delta(t_{0})}\bigl\|\bigl(\dot{x}(t_{0}), \dot{\lambda}(t_{0})\bigr)\bigr\|^{2} + \int_{t_{0}}^{t} \left[ \frac{(\alpha - 1) s^{\alpha - 2} \delta(s) - s^{\alpha - 1}\dot{\delta}(s)}{\delta^{2}(s)}\right] \bigl\|\bigl(\dot{x}(s), \dot{\lambda}(s)\bigr)\bigr\|^{2} ds \nonumber\\
        &\leq I_{3}(t) + \frac{t_{0}^{\alpha - 1}}{\delta(t_{0})}\bigl\|\bigl(\dot{x}(t_{0}), \dot{\lambda}(t_{0})\bigr)\bigr\|^{2} + \frac{\alpha - 1}{t_{0}^{2} \delta(t_{0})}\int_{t_{0}}^{t} s^{\alpha} \bigl\| \bigl( \dot{x}(s), \dot{\lambda}(s)\bigr)\bigr\|^{2} ds. \label{eq: inequality for I_3}
    \end{align}
    
    $\bullet$ \textit{The integral $I_{4}(t)$}. Yet again, integration by parts produces
    \begin{equation*}
        I_{4}(t) = t^{\alpha + 1} \delta(t) \left\lVert A^{*} \left( \lambda \left( t \right) - \lambda_{*} \right) \right\rVert^{2} - t_{0}^{\alpha + 1} \delta(t_{0}) \bigl\| A^{*}(\lambda(t_{0}) - \lambda_{*})\bigr\|^{2} - \alpha \int_{t_{0}}^{t}s^{\alpha} \delta(s) \left\lVert A^{*}(\lambda(s) - \lambda_{*}) \right\rVert^{2} ds,
    \end{equation*}
    and from here
    \begin{equation}\label{eq: inequality for I_4}
        t^{\alpha + 1} \delta(t) \left\lVert A^{*} \left( \lambda \left( t \right) - \lambda_{*} \right) \right\rVert^{2} = I_{4}(t) + t_{0}^{\alpha + 1} \delta(t_{0}) \bigl\| A^{*}(\lambda(t_{0}) - \lambda_{*})\bigr\|^{2} + \alpha \int_{t_{0}}^{t}s^{\alpha} \delta(s) \left\lVert A^{*}(\lambda(s) - \lambda_{*}) \right\rVert^{2} ds.
    \end{equation}
    
    $\bullet$ \textit{The integral $I_{5}(t)$}. Integration by parts entails
    \begin{align*}
        I_{5}(t) &= \int_{t_{0}}^{t} \left\langle \frac{d}{ds}\bigl( s^{\alpha} \dot{x}(s)\bigr), A^{*}(\lambda(s) - \lambda_{*})\right\rangle ds \\
        &= t^{\alpha} \bigl\langle \dot{x}(t), A^{*}(\lambda(t) - \lambda_{*})\bigr\rangle - t_{0}^{\alpha} \bigl\langle \dot{x}(t_{0}), A^{*}(\lambda(t_{0}) - \lambda_{*})\bigr\rangle - \int_{t_{0}}^{t} s^{\alpha} \bigl\langle \dot{x}(s), A^{*}\dot{\lambda}(s)\bigr\rangle ds.
    \end{align*}
    By the Cauchy-Schwarz inequality, we deduce that
    \[
    \int_{t_{0}}^{t} s^{\alpha} \bigl\langle \dot{x}(s), A^{*}\dot{\lambda}(s)\bigr\rangle ds \leq \left\lVert A \right\rVert \int_{t_{0}}^{t} s^{\alpha} \left\lVert \dot{x}(s) \right\rVert \left\lVert \dot{\lambda}(s) \right\rVert ds \leq \dfrac{\left\lVert A \right\rVert}{2} \int_{t_{0}}^{t} s^{\alpha}\Bigl( \bigl\|\dot{x}(s)\bigr\|^{2} + \bigl\| \dot{\lambda}(s)\bigr\|^{2}\Bigr) ds,
    \]
    and thus
    \begin{equation}\label{eq: inequality for I_5}
        \begin{split}
            0 \leq \: &I_{5}(t) - t^{\alpha} \bigl\langle \dot{x}(t), A^{*}(\lambda(t) - \lambda_{*})\bigr\rangle +  \bigl|t_{0}^{\alpha} \bigl\langle \dot{x}(t_{0}), A^{*}(\lambda(t_{0}) - \lambda_{*})\bigr\rangle\bigr| \\
            &+ \dfrac{\left\lVert A \right\rVert}{2} \int_{t_{0}}^{t} s^{\alpha} \bigl\| \bigl( \dot{x}(s), \dot{\lambda}(s)\bigr)\bigr\|^{2} ds.
        \end{split}
    \end{equation}
    Now, to the equality \eqref{eq: inequality for I_4}, we add the inequalities \eqref{eq: inequality for I_1}, \eqref{eq: inequality for I_2}, \eqref{eq: inequality for I_3} and \eqref{eq: inequality for I_5} and then proceed to employ \eqref{eq: inequality for I_1, I_2, I_3, I_4, and I_5}:
    \begin{align}
        &\theta t^{\alpha + 1} \delta(t) \left\lVert A^{*} \left( \lambda \left( t \right) - \lambda_{*} \right) \right\rVert^{2} \nonumber\\
        \leq \: &I_{1}(t) + \theta I_{2}(t) + \alpha I_{3}(t) + \theta I_{4}(t) + 2I_{5} - t^{\alpha} \dot{\varphi}(t) \nonumber\\
        &+\int_{t_{0}}^{t}s^{\alpha} \left[ \theta (\alpha + 1) W(s) + \left( \frac{\alpha(\alpha - 1)}{t_{0}^{2} \delta(t_{0})} + \left\lVert A \right\rVert\right)\right] \bigl\|\bigl(\dot{x}(s), \dot{\lambda}(s)\bigr)\bigr\|^{2} ds \nonumber\\
        &+ \theta \alpha \int_{t_{0}}^{t} s^{\alpha} \delta(s) \left\lVert A^{*}(\lambda(s) - \lambda_{*}) \right\rVert^{2} ds - 2 t^{\alpha} \bigl\langle \dot{x}(t), A^{*}(\lambda(t) - \lambda_{*})\bigr\rangle + C_{5} \nonumber\\
        \leq \: & -t^{\alpha} \dot{\varphi}(t) + \int_{t_{0}}^{t} s^{\alpha} V(s) ds + \int_{t_{0}}^{t} s^{\alpha} \Bigl[ \bigl(\theta(\alpha + 1) - 1\bigr)\delta(t) + \theta s \dot{\delta}(s)\Bigr] \left\lVert A^{*}(\lambda(s) - \lambda_{*}) \right\rVert^{2} ds \nonumber\\
        &-2 t^{\alpha} \bigl\langle \dot{x}(t), A^{*}(\lambda(t) - \lambda_{*})\bigr\rangle + C_{5}, \label{eq: we divide this one by t^alpha}
    \end{align}
    where, for $s\geq t_{0}$, 
    \begin{align*}
        V(s) := \: &\theta (\alpha + 1) W(s) + \left( \frac{\alpha(\alpha - 1)}{t_{0}^{2} \delta(_{0})} + \left\lVert A \right\rVert\right) \bigl\| \bigl( \dot{x}(s), \dot{\lambda}(s)\bigr)\bigr\|^{2} \\
        &+ C_{3} \delta(s) \|\nabla f(x(s)) - \nabla f(x_{*})\|^{2} + C_{4} \delta(s) \|Ax(s) - b\|^{2},
    \end{align*}
    and the constant $C_{5}$ is given by 
    \begin{align*}
        C_{5} := \: &t_{0}^{\alpha} |\dot{\varphi}(t_{0})| + \theta t_{0}^{\alpha + 1}W(t_{0}) + \alpha \frac{t_{0}^{\alpha - 1}}{\delta(t_{0})}\bigl\|\bigl(\dot{x}(t_{0}, \dot{\lambda}(t_{0}))\bigr)\bigr\|^{2} \\
        &+\theta t_{0}^{\alpha + 1} \delta(t_{0}) \bigl\| A^{*}(\lambda(t_{0}) - \lambda_{*})\bigr\|^{2} + 2 t_{0}^{\alpha} \bigl| \bigl\langle \dot{x}(t_{0}), A^{*}(\lambda(t_{0}) - \lambda_{*})\bigr\rangle\bigr| \geq 0.
    \end{align*}
    Now we divide \eqref{eq: we divide this one by t^alpha} by $t^{\alpha}$, thus obtaining
    \begin{align*}
        \theta t \delta(t) \left\lVert A^{*} \left( \lambda \left( t \right) - \lambda_{*} \right) \right\rVert^{2} \leq \: &- \dot{\varphi}(t) + \frac{1}{t^{\alpha}}\int_{t_{0}}^{t} s^{\alpha} V(s) ds \\
        &+ \frac{1}{t^{\alpha}}\int_{t_{0}}^{t} s^{\alpha} \Bigl[ \bigl(\theta(\alpha + 1) - 1\bigr)\delta(s) + \theta s \dot{\delta}(s)\Bigr] \left\lVert A^{*}(\lambda(s) - \lambda_{*}) \right\rVert^{2} ds \\
        &-2 \bigl\langle \dot{x}(t) , A^{*}(\lambda(t) - \lambda_{*})\bigr\rangle + \frac{C_{5}}{t^{\alpha}}.
    \end{align*}
    Now, we integrate this inequality from $t_{0}$ to $r$. We get
    \begin{align}
        &\theta \int_{t_{0}}^{r} t \delta(t) \left\lVert A^{*} \left( \lambda \left( t \right) - \lambda_{*} \right) \right\rVert^{2} dt \nonumber\\
        \leq \  & \varphi(t_{0}) - \varphi(r)  + \int_{t_{0}}^{r}\frac{1}{t^{\alpha}} \left( \int_{t_{0}}^{t} s^{\alpha} V(s) ds\right) dt \nonumber\\
        &\quad + \int_{t_{0}}^{r} \frac{1}{t^{\alpha}} \left(\int_{t_{0}}^{t} s^{\alpha} \Bigl[ \bigl(\theta(\alpha + 1) - 1\bigr)\delta(s) + \theta s \dot{\delta}(s)\Bigr] \left\lVert A^{*}(\lambda(s) - \lambda_{*}) \right\rVert^{2} ds\right) dt \nonumber\\
        &\quad - 2 \int_{t_{0}}^{r} \bigl\langle A\dot{x}(t), \lambda(t) - \lambda_{*}\bigr\rangle dt + C_{5} \int_{t_{0}}^{r} t^{\alpha} dt. \label{eq: double integrals}
    \end{align}
    We now recall some important facts. First of all, we have
    \begin{equation}\label{eq: fact 1}
        \int_{t_{0}}^{r}\frac{1}{t^{\alpha}} dt \leq \frac{1}{(\alpha - 1)t_{0}^{\alpha - 1}}.
    \end{equation}
    In addition, according to Lemma \ref{lem:lemma a2}, it holds
    \begin{equation}\label{eq: fact 2}
    \int_{t_{0}}^{r} \frac{1}{t^{\alpha}} \left(\int_{t_{0}}^{t} s^{\alpha} V(s) ds\right)dt \leq \frac{1}{\alpha - 1}\int_{t_{0}}^{r}t V(t) dt,
    \end{equation}
and
	\begin{align}\label{eq: fact 3}
		& \int_{t_{0}}^{r} \frac{1}{t^{\alpha}} \left(\int_{t_{0}}^{t} s^{\alpha} \Bigl[ \bigl(\theta(\alpha + 1) - 1\bigr)\delta(s) + \theta s \dot{\delta}(s)\Bigr] \left\lVert A^{*}(\lambda(s) - \lambda_{*}) \right\rVert^{2} ds\right) dt \nonumber \\
		\leq \ 	& \frac{1}{\alpha - 1}\int_{t_{0}}^{r} \Bigl[ \bigl(\theta(\alpha + 1) - 1\bigr)\delta(t) + \theta t \dot{\delta}(t)\Bigr] \left\lVert A^{*}(\lambda(t) - \lambda_{*}) \right\rVert^{2} dt,
	\end{align}
respectively.
    \noindent Finally, integrating by parts leads to
    \begin{align}
        &- \int_{t_{0}}^{r} \bigl\langle A \dot{x}(t), \lambda(t) - \lambda_{*}\bigr\rangle dt \nonumber\\
        = \     & - \bigl\langle Ax(r) - b, \lambda(r) - \lambda_{*}\bigr\rangle + \bigl\langle Ax(t_{0}) - b, \lambda(t_{0}) - \lambda_{*}\bigr\rangle + \int_{t_{0}}^{r} \bigl\langle Ax(t) - b, \dot{\lambda}(t)\bigr\rangle dt \nonumber\\
        \leq \  & \|Ax(r) - b\|\| \lambda(r) - \lambda_{*}\| + \left\lVert Ax(t_{0}) - b \right\rVert \|\lambda(t_{0}) - \lambda_{*}\| + \int_{t_{0}}^{r} \bigl\langle Ax(t) - b, \dot{\lambda}(t)\bigr\rangle dt \nonumber\\
        \leq \  & \sup_{t\geq t_{0}} \{\left\lVert Ax \left( t \right) - b \right\rVert\left\lVert \lambda \left( t \right) - \lambda_{*} \right\rVert\} + \left\lVert Ax(t_{0}) - b \right\rVert\| \lambda(t_{0} - \lambda_{*})\| \nonumber\\
        &\quad + \frac{1}{2}\int_{t_{0}}^{r} \bigl( \left\lVert Ax \left( t \right) - b \right\rVert^{2} + \left\lVert \dot{\lambda} \left( t \right) \right\rVert^{2}\bigr) dt. \label{eq: fact 4}
    \end{align}
    The supremum term is finite due to the boundedness of the trajectory. Now, by using the nonnegativity of $\varphi$ and the facts \eqref{eq: fact 1}, \eqref{eq: fact 2}, \eqref{eq: fact 3} and \eqref{eq: fact 4} on \eqref{eq: double integrals}, we come to 
    \begin{align}
        &\dfrac{\theta}{\alpha - 1} \int_{t_{0}}^{r}t \sigma(t) \left\lVert A^{*} \left( \lambda \left( t \right) - \lambda_{*} \right) \right\rVert^{2} dt \nonumber\\
        = \ 	& \int_{t_{0}}^{r} \left[ \theta \delta(t) - \frac{\bigl(\theta(\alpha + 1) - 1\bigr)\delta(t) + \theta t \dot{\delta}(t)}{\alpha - 1}\right] t \left\lVert A^{*} \left( \lambda \left( t \right) - \lambda_{*} \right) \right\rVert^{2} dt \nonumber\\
        \leq \ 	& \frac{1}{\alpha - 1}\int_{t_{0}}^{r} t V(t) dt + \int_{t_{0}}^{r} t \left( \left\lVert Ax \left( t \right) - b \right\rVert^{2} + \left\lVert \dot{\lambda} \left( t \right) \right\rVert^{2} \right) dt + C_{6}, \label{eq: integral for dual trajectory}
    \end{align}
    where 
    \begin{align*}
        C_{6} := \varphi(t_{0}) + 2 \sup_{t\geq t_{0}} \left\lbrace \left\lVert Ax \left( t \right) - b \right\rVert\left\lVert \lambda \left( t \right) - \lambda_{*} \right\rVert \right\rbrace + 2 \left\lVert Ax(t_{0}) - b \right\rVert\| \lambda(t_{0} - \lambda_{*})\| + \frac{C_{5}}{(\alpha - 1)t_{0}^{\alpha - 1}}.
    \end{align*}
    According to \eqref{eq:int estimate 2} and \eqref{eq:int estimate 3} in Theorem \ref{thm: integral estimates}, as well as Lemma \ref{lem: first condition of opial lemma}, we know that the mappings $t \mapsto t V(t)$ and $t \mapsto t \left( \left\lVert Ax \left( t \right) - b \right\rVert^{2} + \left\lVert \dot{\lambda} \left( t \right) \right\rVert \right)$ belong to $\sL^{1}\left[ t_{0}, +\infty \right)$. Therefore, by taking the limit as $r \to +\infty$ in \eqref{eq: integral for dual trajectory} we obtain
    \[
    \int_{t_{0}}^{+\infty} t \sigma(t) \left\lVert A^{*} \left( \lambda \left( t \right) - \lambda_{*} \right) \right\rVert^{2} dt < +\infty.
    \]
    Again, from \eqref{eq:inequality for delta and sigma} we conclude that
    \[
    \int_{t_{0}}^{+\infty} t \delta(t) \left\lVert A^{*} \left( \lambda \left( t \right) - \lambda_{*} \right) \right\rVert^{2} dt \leq \frac{1}{C_{2}} \int_{t_{0}}^{+\infty} t \sigma(t) \left\lVert A^{*} \left( \lambda \left( t \right) - \lambda_{*} \right) \right\rVert^{2} dt < +\infty,
    \]
    which completes the proof. 
\end{proof}
The following result is the final step towards the second condition of Opial's Lemma. 

\begin{thm}\label{thm:rates for gradient and dual variable}
    Let $(x, \lambda) : \left[ t_{0}, +\infty \right) \to \mathcal{X} \times \mathcal{Y}$ be a solution to \eqref{ds:timerescaled} and $(x_{*}, \lambda_{*}) \in \mathbb{S}$. Then it holds 
     \begin{equation} \label{eq:rates for gradient and dual variable}
        \left\lVert \nabla f \left( x \left( t \right) \right) - \nabla f \left( x_{*} \right) \right\rVert = o \left(\frac{1}{\sqrt{t} \sqrt[4]{\delta(t)}}\right) \textrm{ and } \left\lVert A^{*} \left( \lambda \left( t \right) - \lambda_{*} \right) \right\rVert = o \left(\frac{1}{\sqrt{t} \sqrt[4]{\delta(t)}}\right) \textrm{ as } t \to +\infty.
    \end{equation}
    Consequently, 
    \[
    \left\lVert \nabla_{x} \Lag \bigl( x \left( t \right) , \lambda \left( t \right) \bigr) \right\rVert  = \left\lVert \nabla f \left( x \left( t \right) \right) + A^{*} \lambda \left( t \right) \right\rVert = o\left(\frac{1}{\sqrt{t} \sqrt[4]{\delta(t)}}\right) \quad \textrm{as} \quad t\to +\infty, 
    \]
    while, as seen earlier, 
    \[
    \left\lVert \nabla_{\lambda} \Lag \bigl( x \left( t \right) , \lambda \left( t \right) \bigr) \right\rVert = \left\lVert Ax \left( t \right) - b \right\rVert =  \mathcal{O}\left(\frac{1}{t^{2} \delta(t)}\right) \quad \textrm{as} \quad t\to +\infty.
    \]
\end{thm}
\begin{proof}
    We first show the gradient rate. For $t\geq t_{0}$, it holds 
    \begin{align}
        \frac{d}{dt}\Bigl(t \sqrt{\delta(t)} \left\lVert \nabla f \left( x \left( t \right) \right) - \nabla f \left( x_{*} \right) \right\rVert^{2}\Bigr) 
        &= \left( \sqrt{\delta(t)} + \frac{t\dot{\delta}(t)}{2\sqrt{\delta(t)}}\right) \left\lVert \nabla f \left( x \left( t \right) \right) - \nabla f \left( x_{*} \right) \right\rVert^{2} \nonumber\\ 
        & \quad + 2 t \sqrt{\delta(t)} \left\langle \nabla f(x(t)) - \nabla f(x_{*}), \frac{d}{dt}\nabla f(x(t))\right\rangle. \label{eq: improved rate 0}
    \end{align}
    On the one hand, by Assumption \ref{assume:para2}, we can write 
    \begin{equation}\label{eq: improved rate 1}
        \left( \sqrt{\delta(t)} + \frac{t\dot{\delta}(t)}{2\sqrt{\delta(t)}}\right) = \left( \sqrt{\delta(t)} + \frac{\sqrt{\delta(t)}}{2} \cdot \frac{t\dot{\delta}(t)}{\delta(t)}\right) \leq \left(1 + \frac{1 - 2\theta}{2\theta}\right) \sqrt{\delta(t)} = \frac{1}{2\theta} \sqrt{\delta(t)}. 
    \end{equation}
    Since $\delta$ is nondecreasing, for $t \geq t_{0}$ we have $\sqrt{\delta(t)} \geq \sqrt{\delta(t_{0})} > 0$. Set $t_{1} := \max\left\{t_{0}, \frac{1}{\sqrt{t_{0}}}\right\}$. Therefore, for $t \geq t_{1}$ it holds
    \[
    \frac{1}{\sqrt{\delta(t)}} \leq \frac{1}{\sqrt{\delta(t_{0})}} = t_{1} \leq t
    \]
    and thus
    \begin{equation}\label{eq: improved rate 2}
    \sqrt{\delta(t)} \leq t \delta(t).
    \end{equation}
    On the other hand, for every $t \geq t_{1}$ we deduce
    \begin{align}
        2 t \sqrt{\delta(t)} \left\langle \nabla f(x(t)) - \nabla f(x_{*}), \frac{d}{dt}\nabla f(x(t))\right\rangle 
        &= 2 t \left\langle \sqrt{\delta(t)} \bigl[\nabla f(x(t)) - \nabla f(x_{*})\bigr], \frac{d}{dt}\nabla f(x(t))\right\rangle \nonumber\\
        &\leq t\delta(t) \left\lVert \nabla f \left( x \left( t \right) \right) - \nabla f \left( x_{*} \right) \right\rVert^{2} + t\left\|\frac{d}{dt}\nabla f(x(t))\right\|^{2} \nonumber\\
        &\leq t\delta(t) \left\lVert \nabla f \left( x \left( t \right) \right) - \nabla f \left( x_{*} \right) \right\rVert^{2} + \ell^{2} t \bigl\| \dot{x}(t)\bigr\|^{2} \label{eq: improved rate 3},
    \end{align}
    where the last inequality is a consequence of the $\ell$-Lipschitz continuity of $\nabla f$.
    By combining \eqref{eq: improved rate 1}, \eqref{eq: improved rate 2} and \eqref{eq: improved rate 3}, from \eqref{eq: improved rate 0} we assert that for every $t\geq t_{1}$
    \begin{equation*}
        \frac{d}{dt}\Bigl(t \sqrt{\delta(t)} \left\lVert \nabla f \left( x \left( t \right) \right) - \nabla f \left( x_{*} \right) \right\rVert^{2}\Bigr) \leq \left( 1 + \frac{1}{2\theta} \right) t \delta(t) \left\lVert \nabla f \left( x \left( t \right) \right) - \nabla f \left( x_{*} \right) \right\rVert^{2} + \ell^{2} t \bigl\| \dot{x}(t)\bigr\|^{2} .
    \end{equation*}
    According to \eqref{eq:int estimate 3} and \eqref{eq:integral for gradient}, the right hand side of the previous inequality belongs to $\sL^{1} [t_{1}, +\infty)$. Since $\delta$ is nondecreasing, for every $t \geq t_{1}$ we have 
    \[
        \sqrt{\delta(t)} = \sqrt{\delta(t)} \cdot \frac{\sqrt{\delta(t)}}{\sqrt{\delta(t)}} \leq \frac{\delta(t)}{\sqrt{\delta(t_{0})}}, 
    \]
    so 
    \begin{equation}\label{eq:inequality for applying lemma a3}
        \int_{t_{1}}^{+\infty} t \sqrt{\delta(t)} \left\lVert \nabla f \left( x \left( t \right) \right) - \nabla f \left( x_{*} \right) \right\rVert^{2} dt \leq \frac{1}{\sqrt{\delta(t_{0})}} \int_{t_{1}}^{+\infty} t \delta(t) \left\lVert \nabla f \left( x \left( t \right) \right) - \nabla f \left( x_{*} \right) \right\rVert^{2} dt < +\infty,
    \end{equation}
    i.e., the function being differentiated also belongs to $\sL^{1} [t_{1}, +\infty)$. Therefore, Lemma \ref{lem:lemma a3} gives us
    \[
    t \sqrt{\delta(t)} \left\lVert \nabla f \left( x \left( t \right) \right) - \nabla f \left( x_{*} \right) \right\rVert^{2} \to 0 \quad \text{as} \quad t\to +\infty. 
    \]
    Proceeding in the exact same way, for every $t\geq t_{1}$ we have 
    \begin{align*}
        &\frac{d}{dt} \Bigl( t \sqrt{\delta(t)} \bigl\|A^{*}(\lambda(t) - \lambda_{*})\bigr\|^{2}\Bigr) \\
        = \     & \left( \sqrt{\delta(t)} + \frac{t \dot{\delta}(t)}{2\sqrt{\delta(t)}}\right) \bigl\|A^{*}(\lambda(t) - \lambda_{*})\bigr\|^{2} + 2 t \sqrt{\delta(t)} \bigl\langle A A^{*}(\lambda(t) - \lambda_{*}), \dot{\lambda}(t)\bigr\rangle \\
        \leq \  & \left( \frac{1}{2\theta} + \|A\|^{2} \right) t \delta(t) \bigl\|A^{*}(\lambda(t) - \lambda_{*})\bigr\|^{2} + t \left\lVert \dot{\lambda} \left( t \right) \right\rVert^{2}.
    \end{align*}
    According to \eqref{eq:int estimate 3} and \eqref{eq:integral for dual variable}, the right hand side of the previous inequality belongs to $\sL^{1}[t_{1}, +\infty)$. Arguing as in \eqref{eq:inequality for applying lemma a3}, we deduce that the function being differentiated also belongs to $\sL^{1}[t_{1}, +\infty)$. Again applying Lemma \ref{lem:lemma a3}, we come to 
    \[
        t \sqrt{\delta(t)} \bigl\|A^{*}(\lambda(t) - \lambda_{*})\bigr\|^{2} \to 0 \quad \text{as} \quad t\to +\infty. 
    \]
    
    Finally, recalling that $A^{*}\lambda_{*} = -\nabla f(x_{*})$, we deduce from the triangle inequality that
    \begin{align*}
    \left\lVert \nabla_{x} \Lag \bigl( x \left( t \right) , \lambda \left( t \right) \bigr) \right\rVert  &= \left\lVert \nabla f \left( x \left( t \right) \right) + A^{*} \lambda \left( t \right) \right\rVert \\ 
    &\leq \left\lVert \nabla f \left( x \left( t \right) \right) - \nabla f \left( x_{*} \right) \right\rVert + \bigl\|A^{*}(\lambda(t) - \lambda_{*})\bigr\| \nonumber \\
    & = o\left(\frac{1}{\sqrt{t} \sqrt[4]{\delta(t)}}\right) \quad \textrm{as} \quad t\to +\infty,
    \end{align*}
    and the third claim follows.
\end{proof}
\begin{rmk}
    The previous theorem also has its own interest. It tells us that the time rescaling parameter also plays a role in accelerating the rates of convergence for $\left\lVert \nabla f \left( x \left( t \right) \right) - \nabla f \left( x_{*} \right) \right\rVert$ and $\|A^{*}(\lambda(t) - \lambda_{*})\|$ as $t\to +\infty$. Moreover, we deduce from \eqref{eq:rates for gradient and dual variable} that the mapping $(x, \lambda) \mapsto (\nabla f(x), A^{*}\lambda)$ is constant along $\mathbb{S}$, as reported in \cite[Proposition A.4]{BotNguyen}.
\end{rmk}

We now come to the final step and show weak convergence of the trajectories of \eqref{ds:timerescaled} to elements of $\mathbb{S}$.
\begin{thm}
	\label{thm:conv}
     Let $(x, \lambda) : \left[ t_{0}, +\infty \right) \to \mathcal{X} \times \mathcal{Y}$ be a solution to \eqref{ds:timerescaled} and $\left( x_{*} , \lambda_{*} \right) \in \sol$. Then $\bigl(x(t), \lambda(t)\bigr)$ converges weakly to a primal-dual solution of \eqref{intro:pb} as $t\to +\infty$.
\end{thm}
\begin{proof}
    For proving this theorem, we make use of Opial's Lemma (see Lemma \ref{lem:lemma a5}). Lemma \ref{lem: first condition of opial lemma} tells us that $\lim_{t\to +\infty}\left\|\bigl(x(t), \lambda(t)\bigr) - (x_{*}, \lambda_{*})\right\|$ exists for every $(x_{*}, \lambda_{*}) \in \mathbb{S}$, which proves condition (i) of Opial's Lemma.    
    
    In order to show condition (ii), we recall the operator $\mathcal{T}_{\mathcal{L}}$ defined in \eqref{moninclusion} by 
    \[
        \TL \left( x , \lambda \right)
        = 
        \begin{pmatrix}
            \nabla f \left( x \right) + A^{*} \lambda \\ b-Ax
        \end{pmatrix}
        \quad 
        \forall (x, \lambda) \in \mathcal{X} \times \mathcal{Y}.
    \]
    Fix $(x_{*}, \lambda_{*}) \in \mathbb{S}$ (in other words, $\mathcal{T}_{\mathcal{L}}(x_{*}, \lambda_{*}) = 0$) and take $(\tilde{x}, \tilde{\lambda})$ any weak sequential cluster point of $\bigl( x(t), \lambda(t)\bigr)$ as $t\to +\infty$, which means there exists a strictly increasing sequence $(t_{n})_{n\in \mathbb{N}} \subseteq [t_{0}, +\infty)$ such that 
    \[
        \bigl( x(t_{n}), \lambda(t_{n})\bigr) \rightharpoonup (\tilde{x}, \tilde{\lambda}) \quad \text{as} \quad n \to +\infty.       
    \]
    Given that $\delta$ is nondecreasing on $\left[ t_{0}, +\infty \right)$, from Theorem \ref{thm:rates for gradient and dual variable} and \eqref{eq:feasibility condition} we deduce that 
    \begin{align*}
        \nabla f(x(t_{n})) - \nabla f(x_{*}) &\to 0, \\
        A^{*}\lambda(x(t_{n})) - A^{*}\lambda_{*} &\to 0, \\
        Ax(t_{n}) - b &\to 0
    \end{align*}
    as $n\to +\infty$. Since $A^{*}\lambda_{*} = -\nabla f(x_{*})$, the previous three statements imply
    \[
        \mathcal{T}_{\mathcal{L}}(x(t_{n}), \lambda(t_{n})) \to 0 \quad \text{as} \quad n \to +\infty.
    \]
    Since we already had $(x(t_{n}), \lambda(t_{n})) \rightharpoonup (\tilde{x}, \tilde{\lambda})$ as $n\to +\infty$, and the graph of $\mathcal{T}_{\mathcal{L}}$ is sequentially closed in $(\mathcal{X} \times \mathcal{Y})^{\text{weak}} \times (\mathcal{X} \times \mathcal{Y})^{\text{strong}}$ (see \cite[Proposition 20.38(ii)]{Bauschke-Combettes:book}), we finally come to 
    \[
        \mathcal{T}_{\mathcal{L}}(\tilde{x}, \tilde{\lambda}) = 0, 
    \]
    which is to say, $(\tilde{x}, \tilde{\lambda}) \in \mathbb{S}$. The proof is thus concluded. 
\end{proof}

\begin{rmk}
	In case $A := 0$ and $b := 0$, the optimization problem \eqref{intro:pb} reduces to the unconstrained optimization problem
	\begin{equation}
		\label{intro:pb-un}
		\min\limits_{x \in \sX} f \left( x \right) .
	\end{equation}
	
	Indeed, the system of optimality conditions \eqref{intro:opt-Lag} read in this case
	\begin{equation*}
		\left( x_{*} , \lambda_{*} \right) \in \sol
		\Leftrightarrow \nabla f \left( x_{*} \right) = 0 \textrm{ and } \lambda_{*} \in \sY,
	\end{equation*}
	in particular, $x_{*} \in \sX$ is an optimal solution of \eqref{intro:pb-un} if and only if $\nabla f \left( x_{*} \right) = 0$.
	The system \eqref{ds:timerescaled} becomes
	\begin{equation*}
		\begin{dcases}
			\ddot{x} \left( t \right) + \dfrac{\alpha}{t} \dot{x} \left( t \right) + \delta \left( t \right) \nabla f \left( x \left( t \right) \right)    		& = 0 \\
			\ddot{\lambda} \left( t \right) + \dfrac{\alpha}{t} \dot{\lambda} \left( t \right)	& = 0 \\
			\Bigl( x \left( t_{0} \right) , \lambda \left( t_{0} \right) \Bigr) 			= \Bigl( x_{0} , \lambda_{0} \Bigr) \textrm{ and }
			\Bigl( \dot{x} \left( t_{0} \right) , \dot{\lambda} \left( t_{0} \right) \Bigr) = \Bigl( \dot{x}_{0} , \dot{\lambda}_{0} \Bigr)
		\end{dcases}.
	\end{equation*}
	The dynamical system in $x$ is reads
	\begin{equation*}
		\begin{dcases}
			\ddot{x} \left( t \right) + \dfrac{\alpha}{t} \dot{x} \left( t \right) + \nabla f \left( x \left( t \right) \right)   = 0 \\
			x(t_0) = x_0 \ \textrm{ and } \dot{x}(t_0) = \dot{x}_0
		\end{dcases},
	\end{equation*}
	for $\alpha \geq 3$, and is nothing else than Nesterov's accelerated gradient system. The trajectory generated by the system in $\lambda$ is $\lambda(t) = \frac{\dot{\lambda}_0 t_0^\alpha}{1-\alpha} t^{1-\alpha} + \lambda_0 - \frac{\dot{\lambda}_0 t_0}{1-\alpha}$ for every $t \geq t_{0}$. The parameters $\theta$ and $\beta$ plays no role in the system. Therefore, the condition on $\delta \left( t \right)$ now becomes
	\begin{equation*}
		\sup_{t\geq t_{0}}\frac{t \dot{\delta}(t)}{\delta(t)} \leq \alpha - 3 ,
	\end{equation*}
which is exactly the one imposed in \cite{Attouch-Chbani-Riahi:SIOPT}, see also \cite{Attouch-Chbani-Riahi:HAL}.
	
	If $\alpha \geq 3$, then Theorem \ref{thm:rates of convergence} (iii) says that $f(x(t))$ converges to $f_*$ with a rate of convergence of $\bO \left( \frac{1}{t^{2} \delta \left( t \right)} \right)$ as $t \to + \infty$, which is the rate derived in \cite{Attouch-Chbani-Riahi:SIOPT,Attouch-Chbani-Riahi:HAL}.
	
	If $\alpha > 3$, then Theorem \ref{thm:conv} gives that the trajectory $x(t)$ converges weakly to an optimal solution of \eqref{intro:pb-un}, as $t \to + \infty$. Notice that no trajectory convergence has been reported in \cite{Attouch-Chbani-Riahi:SIOPT,Attouch-Chbani-Riahi:HAL}.
	
	Finally, we mention that the convergence of the trajectory in the critical case $\alpha = 3$ is still an open question, even for non time rescaling case (\cite{Attouch-Chbani-Peypouquet-Redont,Su-Boyd-Candes}), as it is the convergence of the iterates of the original Nesterov's acceleration algorithm (\cite{FISTA,Nesterov:83}).
\end{rmk}


\section{Numerical experiments}

We will illustate the theoretical results by two numerical examples, with $\mathcal{X} = \mathbb{R}^{4}$ and $\mathcal{Y} = \mathbb{R}^{2}$. We will address two minimization problems with linear constraints; one with a strongly convex objective function and another with a convex objective function which is not strongly convex. In both cases, the linear constraints are dictated by 
\[
    A = 
    \begin{bmatrix}
        1 & -1 & -1 & 0 \\
        0 & 1 & 0 & -1
    \end{bmatrix}
    \qquad 
    \text{and}
    \qquad 
    b = 
    \begin{bmatrix}
        0 \\
        0
    \end{bmatrix}.
\]
\begin{ex}
\label{ex:strong}
Consider the minimization problem 
\begin{equation*}
    \begin{array}{rl}
	\min & f(x_{1}, x_{2}, x_{3}, x_{4}) := (x_{1} - 1)^{2} + (x_{2} - 1)^{2} + x_{3}^{2} + x_{4}^{2} \\
	\textrm{subject to} 	& x_{1} - x_{2} - x_{3} = 0 \\
	& x_{2} - x_{4} = 0.
\end{array}
\end{equation*}
The optimality conditions can be calculated and lead to the following primal-dual solution pair
\[
x_{*} = 
\begin{bmatrix}
	0.8 \\
	0.6 \\
	0.2 \\
	0.6
\end{bmatrix}
\qquad
\text{and}
\qquad 
\lambda_{*} = 
\begin{bmatrix}
	0.4 \\
	1.2
\end{bmatrix}.
\]
\end{ex}

\begin{ex}
\label{ex:notstrong}
Consider the minimization problem
\begin{equation*}
    \begin{array}{rl}
	\min & f(x_{1}, x_{2}, x_{3}, x_{4}) := \log\left( 1 + e^{-x_{1} - x_{2}}\right) + x_{3}^{2} + x_{4}^{2} \\
	\textrm{subject to} 	& x_{1} - x_{2} - x_{3} = 0 \\
	& x_{2} - x_{4} = 0.
\end{array}
\end{equation*}
This problem is similar to the regularized logistic regression frequently used in machine learning.
We cannot explicitly calculate the optimality conditions as in the previous case; instead, we use the last solution in the numerical experiment as the approximate solution.
\end{ex}
To comply with Assumption \ref{assume:para2}, we choose $t_{0} > 0$, $\alpha = 8$, $\beta = 10$, $\theta = \frac{1}{6}$, and we test four different choices for the rescaling parameter: $\delta(t) = 1$ (i.e., the (PD-AVD) dynamics in \cite{Zeng-Lei-Chen,BotNguyen}), $\delta(t) = t$, $\delta(t) = t^{2}$ and $\delta(t) = t^{3}$. In both examples, the initial conditions are 
\[
    x(t_{0}) = 
    \begin{bmatrix}
        0.5 \\
        0.5 \\
        0.5 \\
        0.5
    \end{bmatrix}, 
    \quad 
    \lambda(t_{0}) = 
    \begin{bmatrix}
        0.2 \\
        0.2
    \end{bmatrix}, 
    \quad 
    \dot{x}(t_{0}) = 
    \begin{bmatrix}
        0.5 \\
        0.5 \\
        0.5 \\
        0.5
    \end{bmatrix}
    \quad 
    \text{and}
    \quad 
    \dot{\lambda}(t_{0}) = 
    \begin{bmatrix}
        0.5 \\
        0.5
    \end{bmatrix}.
\]
For each choice of $\delta$, we plot, using a logarithmic scale, the primal-dual gap $\mathcal{L}\bigl(x(t), \lambda_{*}\bigr) - \mathcal{L}\bigl(x_{*}, \lambda(t)\bigr)$, the feasibility measure $\left\lVert A x \left( t \right) - b \right\rVert$ and  the functional values $\left\lvert f \left( x \left( t \right) \right) - f_{*} \right\rvert$, to highlight the theoretical result in Theorem \ref{thm:rates of convergence}.
We also illustrate the findings from Theorem \ref{thm:rates for gradient and dual variable}, namely, we plot the quantities $\left\lVert \nabla f \left( x \left( t \right) \right) - \nabla f \left( x_{*} \right) \right\rVert$ and $\left\lVert A^{*} \left( \lambda \left( t \right) - \lambda_{*} \right) \right\rVert$, as well as the velocity $\Vert ( \dot{x} (t) , \dot{\lambda} (t) ) \Vert$. 

Figure \ref{fig:strong} and \ref{fig:notstrong} display these plots for Example \ref{ex:strong} and \ref{fig:notstrong}, respectively. 
As predicted by the theory, choosing faster-growing time rescaling parameters yields better convergence rates. This is not the case for the velocities.

\begin{figure}[!htb]	
	\minipage{0.33\textwidth}
	\includegraphics[width=\linewidth]{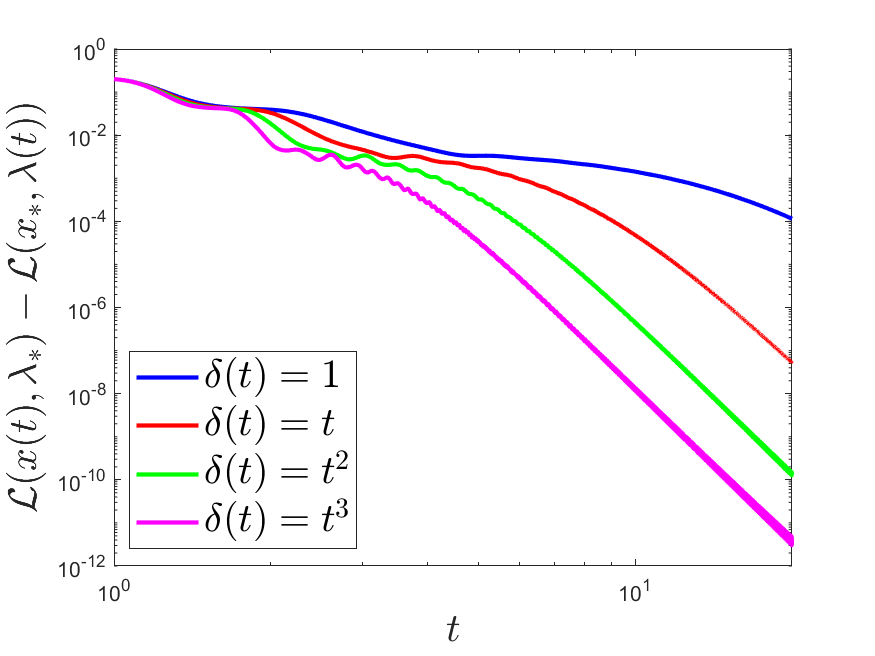}
	\endminipage\hfill
	\minipage{0.33\textwidth}
	\includegraphics[width=\linewidth]{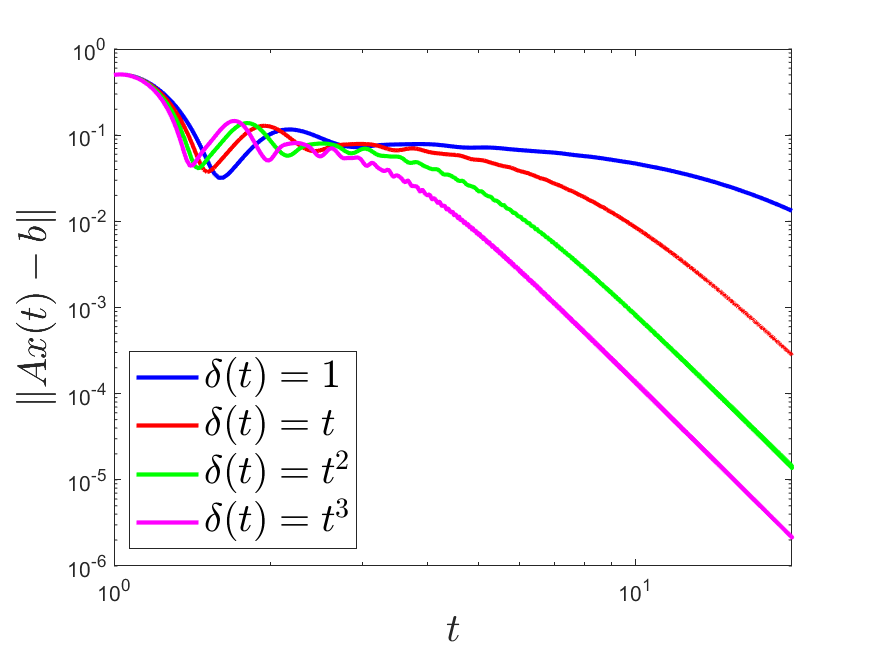}
	\endminipage\hfill
	\minipage{0.33\textwidth}
	\includegraphics[width=\linewidth]{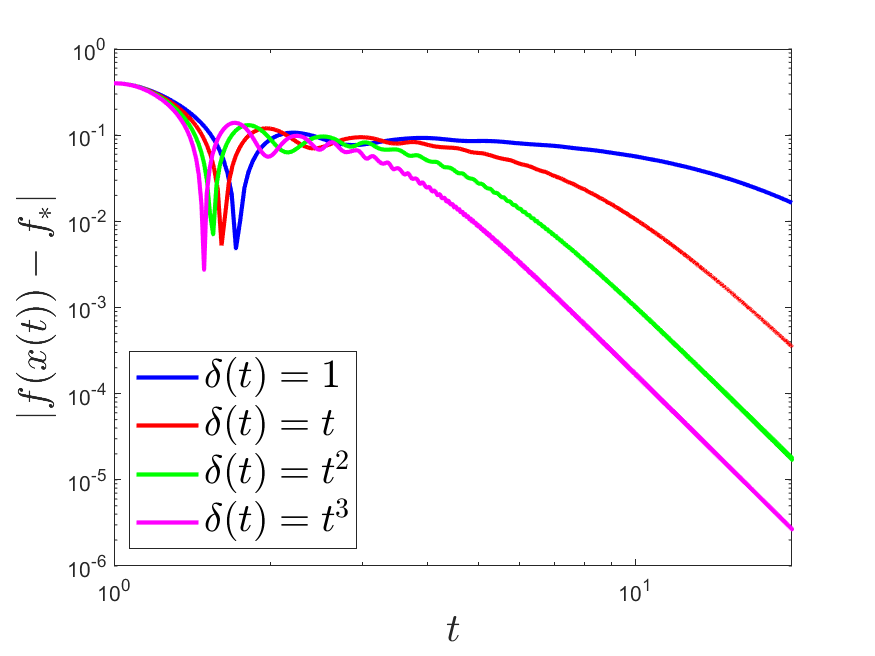}
	\endminipage\hfill
	\minipage{0.33\textwidth}
	\includegraphics[width=\linewidth]{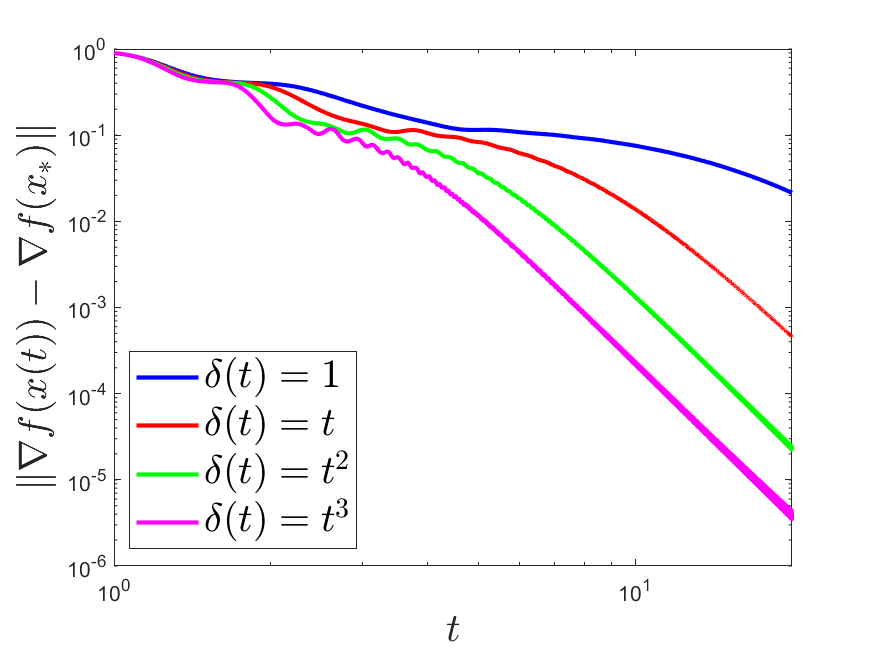}
	\endminipage\hfill
	\minipage{0.33\textwidth}
	\includegraphics[width=\linewidth]{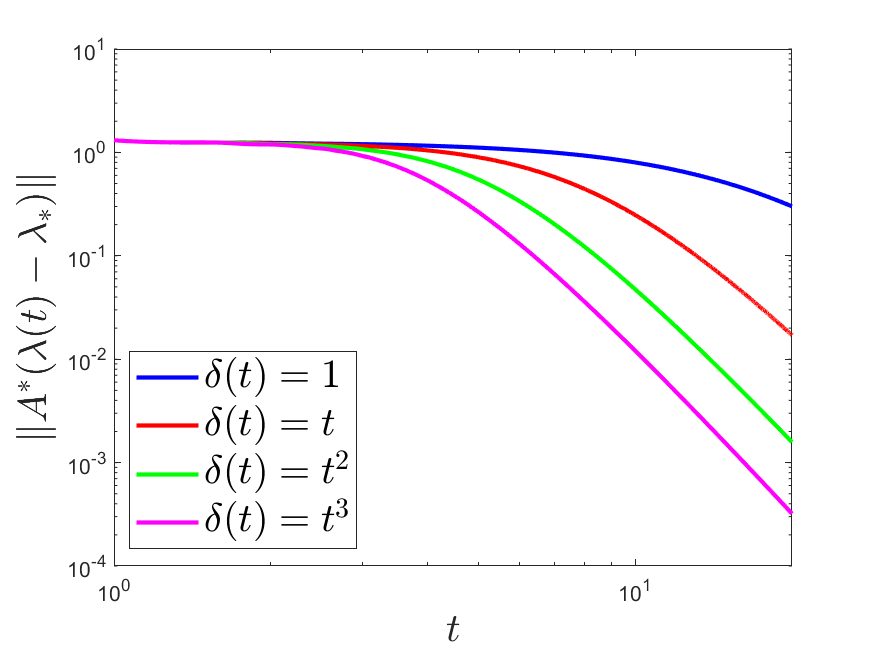}
	\endminipage\hfill
	\minipage{0.33\textwidth}
	\includegraphics[width=\linewidth]{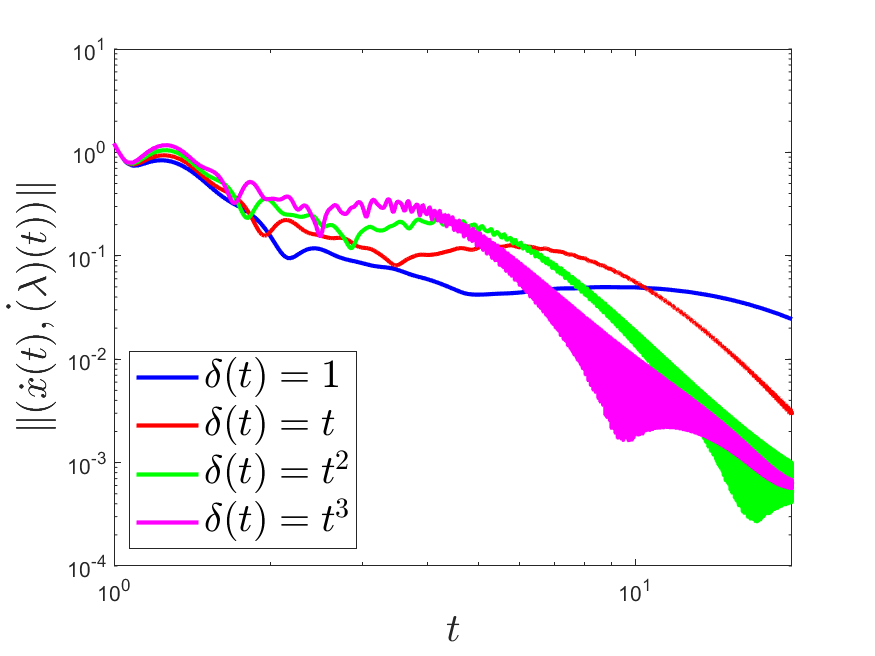}
	\endminipage	
	\caption{The function $\delta \left( t \right)$ influences convergence behaviour in Example \ref{ex:strong}}\label{fig:strong}	
\end{figure}

\begin{figure}[!htb]
	\minipage{0.33\textwidth}
	\includegraphics[width=\linewidth]{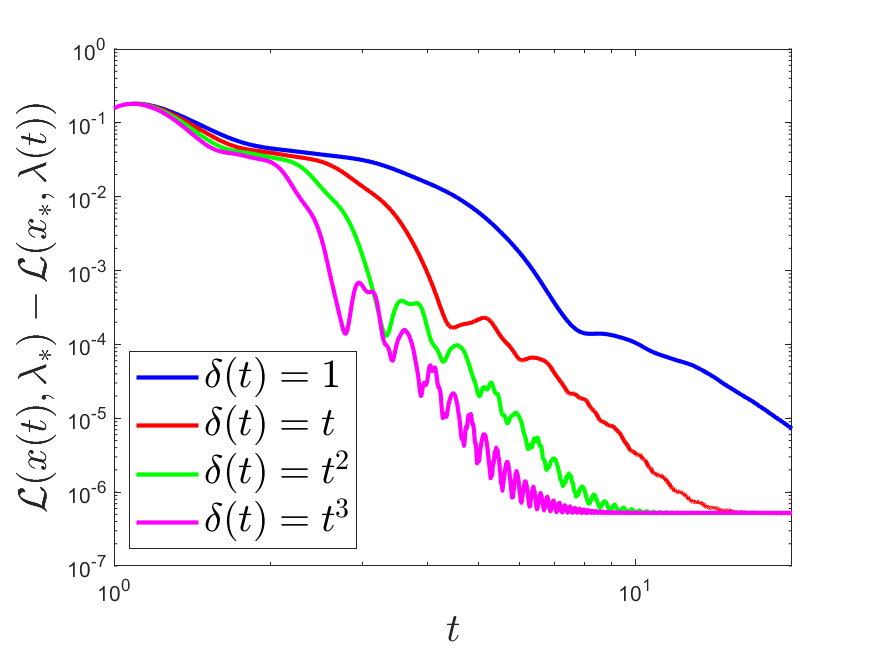}
	\endminipage\hfill
	\minipage{0.33\textwidth}
	\includegraphics[width=\linewidth]{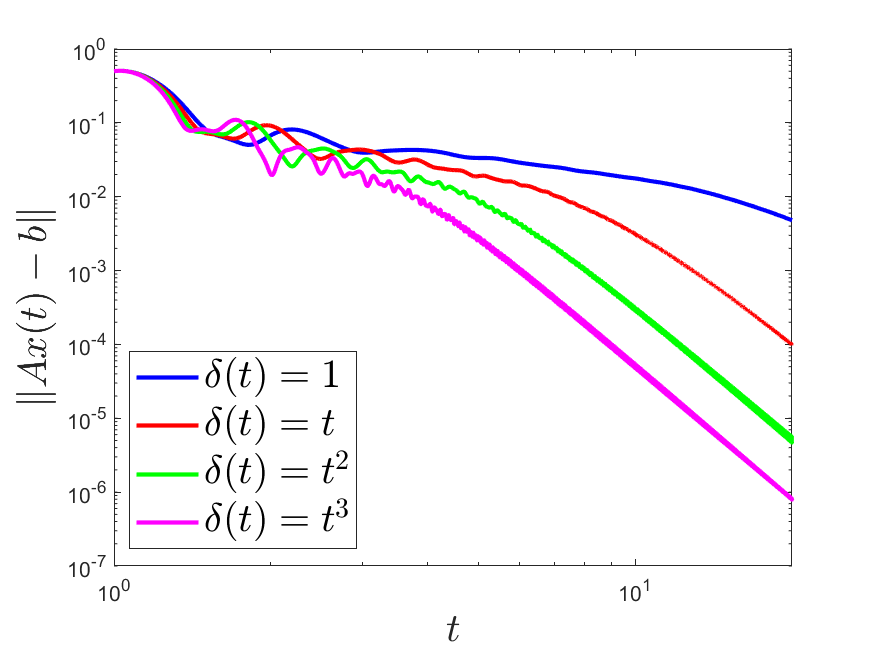}
	\endminipage\hfill
	\minipage{0.33\textwidth}
	\includegraphics[width=\linewidth]{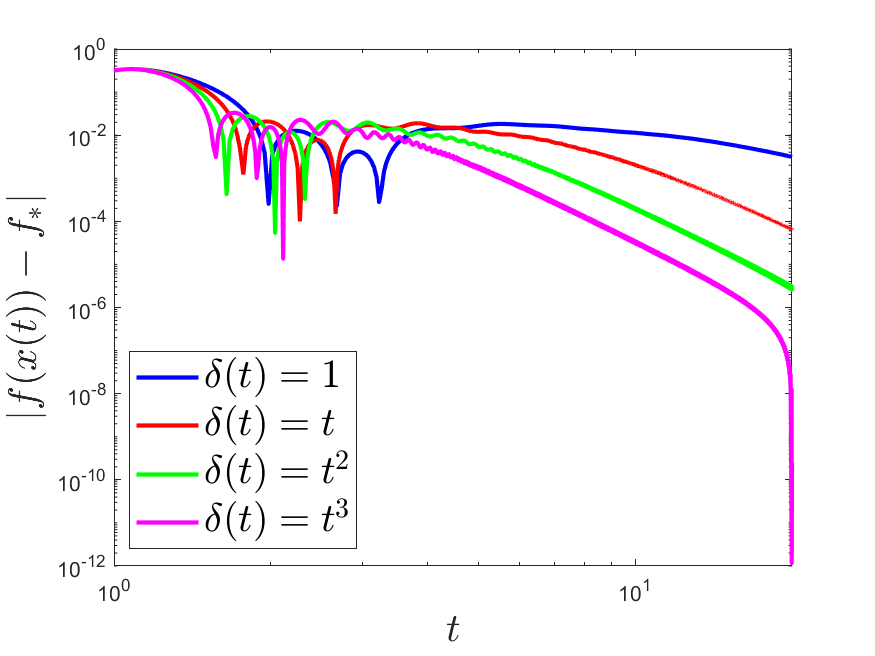}
	\endminipage\hfill
	\minipage{0.33\textwidth}
	\includegraphics[width=\linewidth]{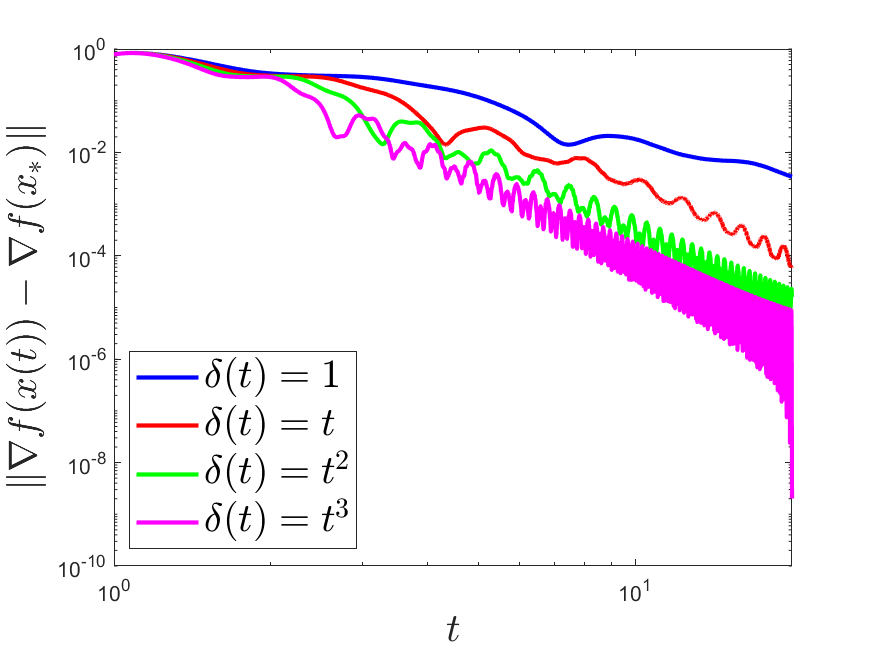}
	\endminipage\hfill
	\minipage{0.33\textwidth}
	\includegraphics[width=\linewidth]{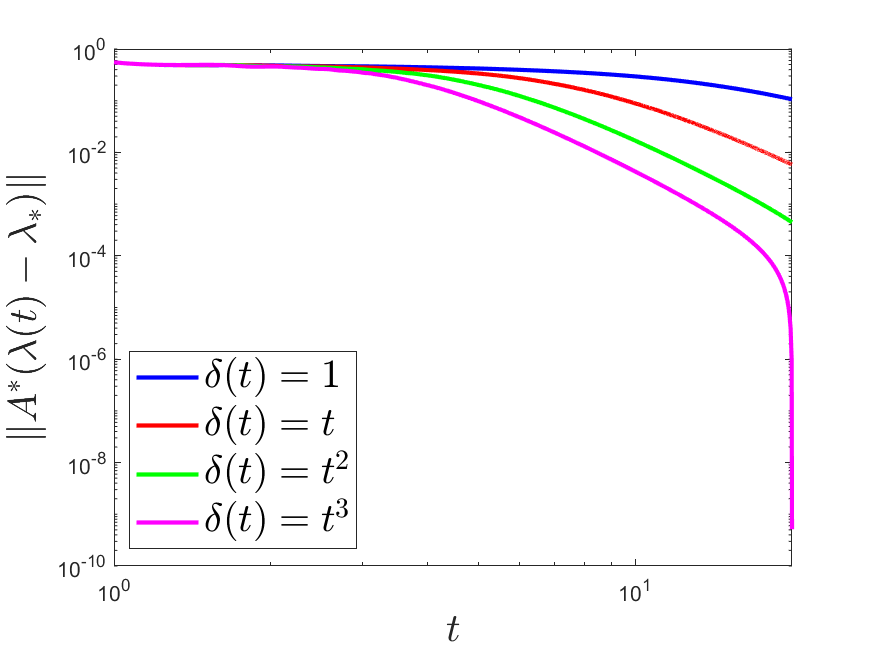}
	\endminipage\hfill
	\minipage{0.33\textwidth}
	\includegraphics[width=\linewidth]{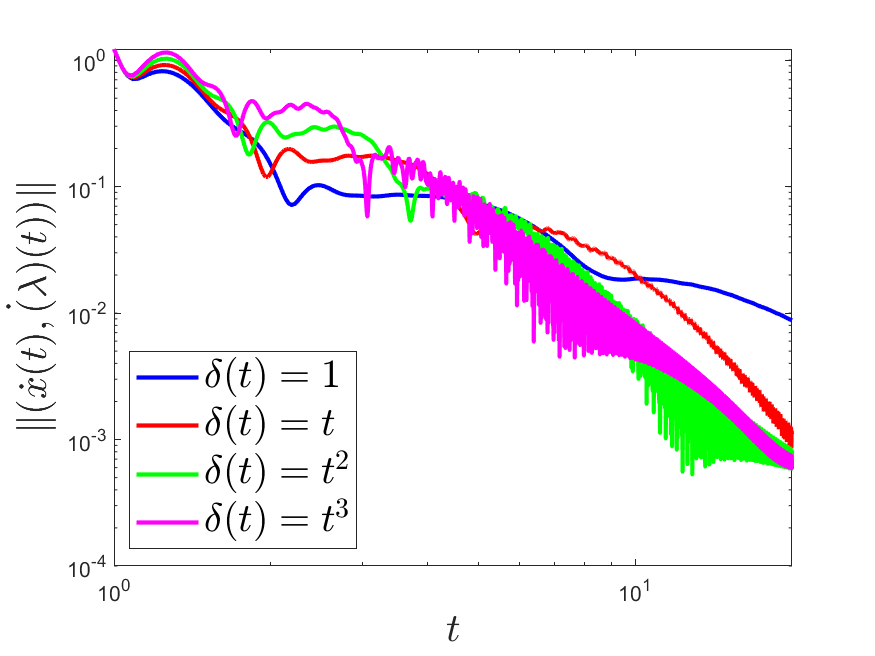}
	\endminipage
	\caption{The function $\delta \left( t \right)$ influences convergence behaviour in Example \ref{ex:notstrong}}\label{fig:notstrong}	
\end{figure}

\pagebreak

\appendix



\section{Appendix}

Here we collect the auxiliary results that are required to carry out many steps in out analysis. 

A proof for the following lemma in the finite-dimensional case can be found in \cite[Lemma 6]{He-Hu-Fang:automatica}. The proof for the infinite-dimensional case is short and virtually identical, so we include it here for the sake of completeness.. 

\begin{lem} \label{lem:lemma a1}
    Assume that $t_{0} > 0$, $g : \left[ t_{0}, +\infty \right) \to \mathcal{Y}$ is a continuous differentiable function, $a : \left[ t_{0}, +\infty \right) \to [0, +\infty)$ is a continuous function, and $C \geq 0$. If, in the sense of Bochner integrability, we have  
    \begin{equation}\label{eq:boundedness hypothesis}
        \left\| g(t) + \int_{t_{0}}^{t} a(s) g(s) ds\right\| \leq C \quad \forall t\geq t_{0}
    \end{equation}
    then 
    \[
        \sup_{t\geq t_{0}}\|g(t)\| \leq 2C < + \infty.
    \]
\end{lem}
\begin{proof}
    Define, for every $t\geq t_{0}$, 
    \[
        G(t) := e^{\int_{t_{0}}^{t} a(s) ds} \int_{t_{0}}^{t} a(s) g(s) ds.
    \]
    Fix $t\geq t_{0}$. The time derivative of $G$ reads 
    \[
        \dot{G}(t) = a(t) e^{\int_{t_{0}}^{t} a(s) ds} \int_{t_{0}}^{t} a(s) g(s) ds + e^{\int_{t_{0}}^{t} a(s) ds} a(t) g(t) 
        = a(t) e^{\int_{t_{0}}^{t} a(s) ds} \left[ g(t) + \int_{t_{0}}^{t} a(s) g(s) ds\right], 
    \]
    so by using \eqref{eq:boundedness hypothesis} and the previous equality we arrive at 
    \begin{equation} \label{eq:derivative of G}
        \bigl\| \dot{G}(t)\bigr\| \leq C a(t) e^{\int_{t_{0}}^{t} a(s) ds} = C \frac{d}{dt}\left( e^{\int_{t_{0}}^{t} a(s) ds}\right).
    \end{equation}
    Since $G(t_{0}) = 0$, we have 
    \[
        G(t) = G(t) - G(t_{0}) = \int_{t_{0}}^{t} \dot{G}(s) ds,
    \]
    so by employing \eqref{eq:derivative of G} and the previous equality we obtain, for every $t\geq t_{0}$, 
    \begin{align*}
        e^{\int_{t_{0}}^{t} a(s) ds} \left\| \int_{t_{0}}^{t} a(s) g(s) ds\right\| &= \|G(t)\| \leq \int_{t_{0}}^{t}\bigl\| \dot{G}(s)\bigr\| ds \leq C \int_{t_{0}}^{t} \frac{d}{ds}\left(e^{\int_{t_{0}}^{s} a(\tau) d\tau}\right) ds \\
        &\leq C \left[ e^{\int_{t_{0}}^{t} a(s) ds} - 1\right] 
        \leq C e^{\int_{t_{0}}^{t} a(s) ds}. 
    \end{align*}
    Dividing both sides of the previous inequality by $e^{\int_{t_{0}}^{t} a(s) ds}$ gives us 
    \begin{equation}\label{eq:boundedness hypothesis 2}
        \left\| \int_{t_{0}}^{t} a(s) g(s) ds\right\| \leq C \quad \forall t\geq t_{0}.
    \end{equation}
    Now, by putting \eqref{eq:boundedness hypothesis} and \eqref{eq:boundedness hypothesis 2} we finally come to 
    \[
        \|g(t)\| \leq \left\| g(t) + \int_{t_{0}}^{t} a(s) g(s) ds\right\| + \left\| \int_{t_{0}}^{t} a(s) g(s) ds\right\| \leq 2C \quad \forall t\geq t_{0},
    \]
    which leads to the announced statement.
\end{proof}

The proofs for the following results can be found in \cite[Lemma A.1]{BotNguyen} and \cite[Lemma 5.2]{Abbas-Attouch-Svaiter}, respectively.
\begin{lem}\label{lem:lemma a2}
    Let $0 < t_{0} \leq r \leq +\infty$ and $h :\left[ t_{0}, +\infty \right) \to [0, +\infty)$ be a continuous function. For every $\alpha > 1$ it holds 
    \[
        \int_{t_{0}}^{r} \frac{1}{t^{\alpha}} \left[ \int_{t_{0}}^{t}s^{\alpha - 1} h(s)\right]dt \leq \frac{1}{\alpha - 1} \int_{t_{0}}^{r} h(t) dt.
    \]
    If $r = +\infty$, then equality holds. 
\end{lem}
%

\begin{lem} \label{lem:lemma a3}
    Let $t_{0} > 0$, $1 \leq p < +\infty$ and $1\leq q \leq +\infty$. Suppose that $F \in \sL^{p}\left[ t_{0}, +\infty \right)$ is a locally absolutely continuous nonnegative function, $G \in \sL^{q}\left[ t_{0}, +\infty \right)$ and 
    \[
        \dot{F}(t) \leq G(t) \qquad \forall t\geq t_{0}. 
    \]
    Then, $\lim_{t\to +\infty} F(t) = 0$. 
\end{lem}

The following lemma is a slight variation of results already present in the literature. See, for example, \cite[Lemma A.2]{Attouch-Chbani-Riahi}. 

\begin{lem} \label{lem:lemma a4}
    Let $t_{0} > 0$, $\alpha > 1$, and let $\phi : [t_{0}, +\infty) \to \mathbb{R}$ be a twice continuously differentiable function bounded from below. Furthermore, assume $w : [t_{0}, +\infty) \to [0, +\infty)$ to be a continuously differentiable function such that $t\mapsto t w(t)$ belongs to $\mathbb{L}^{1}[t_{0}, +\infty)$ and
    \[
        t\ddot{\phi}(t) + \alpha \dot{\phi}(t) + t^{2}\dot{w}(t) \leq 0 \qquad \forall t\geq t_{0}.
    \]
    Then, the positive part $[\dot{\phi}]_{+}$ of $\dot{\phi}$ belongs to $\mathbb{L}^{1}[t_{0}, +\infty)$ and the limit $\lim_{t\to +\infty} \phi(t)$ is a real number. 
\end{lem}

\begin{proof}
    Fix $t\geq t_{0}$. Adding $(\alpha + 1)t w(t)$ to both sides of the previous inequality and then multiplying it by $t^{\alpha - 1}$ yields 
    \[
        \frac{d}{dt}\bigl(t^{\alpha} \dot{\phi}(t)\bigr) + \frac{d}{dt}\bigl(t^{\alpha + 1}w(t)\bigr) \leq (\alpha + 1)t^{\alpha}w(t).
    \]
    Since the previous inequality holds for any $t\geq t_{0}$, we can integrate it from $t_{0}$ to $t\geq t_{0}$ to get 
    \[
        t^{\alpha} \dot{\phi}(t) - t_{0}^{\alpha} \dot{\phi}(t_{0}) + t^{\alpha + 1}w(t) - t_{0}^{\alpha + 1}w(t_{0}) \leq (\alpha + 1) \int_{t_{0}}^{t} s^{\alpha}w(s) ds.
    \]
    After dropping the nonnegative term $t^{\alpha + 1}w(t)$ and dividing by $t^{\alpha}$ we arrive at 
    \[
        \dot{\phi}(t) \leq \frac{\widetilde{C}}{t^{\alpha}} + \frac{\alpha + 1}{t^{\alpha}} \int_{t_{0}}^{t}s^{\alpha} w(s) ds \qquad \forall t\geq t_{0}, 
    \]
    where 
    \[
        \widetilde{C} := t_{0}^{\alpha} \bigl| \dot{\phi}(t_{0})\bigr| + t_{0}^{\alpha + 1}w(t_{0}),
    \]
    which further leads to 
    \[
        [\dot{\phi}(t)]_{+} \leq \frac{\widetilde{C}}{t^{\alpha}} + \frac{\alpha + 1}{t^{\alpha}} \int_{t_{0}}^{t}s^{\alpha}w(s) ds \qquad \forall t\geq t_{0}.
    \]
    Now, we integrate this inequality from $t_{0}$ to $r\geq t_{0}$ and we apply Lemma \ref{lem:lemma a2} with $h : [t_{0}, +\infty) \to [0, +\infty)$ given by $h(s) := sw(s)$ to obtain 
    \begin{align*}
        \int_{t_{0}}^{r}[\dot{\phi}(t)]_{+}dt &\leq \widetilde{C}\int_{t_{0}}^{r} \frac{1}{t^{\alpha}}dt + (\alpha + 1) \int_{t_{0}}^{r}\frac{1}{t^{\alpha}} \left[ \int_{t_{0}}^{t} s^{\alpha - 1}\cdot sw(s)ds\right]dt \\
        &\leq \frac{\widetilde{C}}{1 - \alpha}\left( \frac{1}{t_{0}^{\alpha - 1}} - \frac{1}{r^{\alpha - 1}}\right) + \frac{\alpha + 1}{1 - \alpha} \int_{t_{0}}^{r}tw(t) dt.
    \end{align*}
    By hypothesis, as $r\to +\infty$, the right hand side of the previous inequality is finite. In other words, 
    \[
        \int_{t_{0}}^{+\infty}[\dot{\phi}(t)]_{+}dt < +\infty.
    \]
    The previous statement, together with the fact that we assumed that $\phi$ was bounded from below, allow us to deduce that the function $\psi : [t_{0}, +\infty) \to \mathbb{R}$ given by 
    \[
        \psi(t) := \phi(t) - \int_{t_{0}}^{t}[\dot{\phi}(s)]_{+}ds
    \]
    is also bounded from below. An easy computation shows that $\dot{\psi}$ is nonpositive on $[t_{0}, +\infty)$, thus $\psi$ is nonincreasing on $[t_{0}, +\infty)$. These facts imply that $\lim_{t\to +\infty}\psi(t)$ is a real number. Finally, we conclude that 
    \[
        \lim_{t\to +\infty} \phi(t) = \lim_{t\to +\infty}\psi(t) + \int_{t_{0}}^{+\infty}[\dot{\phi}(s)]_{+}ds \: \in\: \mathbb{R}.
    \]
\end{proof}

The proof for Opial's Lemma can be found in \cite{Opial}.

\begin{lem}[Opial's Lemma] \label{lem:lemma a5}
    Let $\mathcal{H}$ be a real Hilbert space, $S \subseteq \mathcal{H}$ a nonempty set, $t_{0} > 0$ and $z : \left[ t_{0}, +\infty \right) \to \mathcal{H}$ a mapping that satisfies
    \begin{enumerate}[label = (\roman*)]
        \item for every $z_{*} \in S$, $\lim_{t\to +\infty} \|z(t) - z_{*}\|$ exists;
        \item every weak sequential cluster point of the trajectory $z(t)$ as $t \to +\infty$ belongs to $S$. 
    \end{enumerate}
    Then, $z(t)$ converges weakly to an element of $S$ as $t \to +\infty$.
\end{lem}

%
%
%
%

\end{document}